\documentclass[11pt]{article}
\usepackage[utf8]{inputenc}
\usepackage{a4wide}
\usepackage{amsmath}
\usepackage{amssymb}
\usepackage{amsthm}

\usepackage[dvips]{graphicx}

\newcommand{\dfn}[1]{\textit{#1}} %definitions
\newcommand{\ex}[1]{e\left(#1\right)} %expansion of \Delta
\newcommand{\vd}[1]{deg(#1)} %vector degree

\newcommand{\gen}[1]{\operatorname{gen}(#1)}
\newcommand{\linspan}[1]{\operatorname{span}(#1)}

\def\BZ{\mathbb Z}

\def\E{\mathcal E}

\def\E{\mathcal E}

\theoremstyle{plain} \newtheorem{thm}{Theorem}[section]
\theoremstyle{plain} \newtheorem{prop}[thm]{Proposition}
\theoremstyle{plain} \newtheorem{lem}[thm]{Lemma}
\theoremstyle{plain} \newtheorem{cor}[thm]{Corollary}
\theoremstyle{plain} 
\theoremstyle{plain} \newtheorem{clm}[thm]{Claim}
\theoremstyle{definition} \newtheorem{definition}[thm]{Definition}

\begin{document}
%\begin{frontmatter}
%opening
\title{Geometric representations of linear codes}
\author{Pavel Rytíř\thanks{Supported by the Czech Science Foundation under the contract no. 201/09/H057 and by the grant \mbox{SVV-2010-261313} (Discrete Methods and Algorithms).}\\ {\small Department of Applied Mathematics,} {\small Charles University in Prague,}\\ {\small Malostranské~náměstí~25}, {\small Prague 118 00, Czech Republic}}
\maketitle
%\author{Pavel Rytíř\fnref{fn1}}
%\ead{rytir@kam.mff.cuni.cz}
%\address{Department of Applied Mathematics, Charles University in Prague, Malostranské~náměstí~25, Prague 118 00, Czech Republic}

%\fntext[fn1]{Supported by the Czech Science Foundation under the contract no. 201/09/H057}

\begin{abstract}
We say that a linear code $\mathcal{C}$ over a field $\mathbb{F}$ is triangular representable if there exists a two dimensional simplicial complex $\Delta$ such that $\mathcal{C}$ is a punctured code of the kernel $\ker \Delta$ of the incidence matrix of $\Delta$ over $\mathbb{F}$ and there is a linear mapping between $\mathcal{C}$ and $\ker \Delta$ which is a bijection and maps minimal codewords to minimal codewords. We show that the linear codes over rationals and over $GF(p)$, where $p$ is a prime, are triangular representable. In the case of finite fields, we show that this representation determines the weight enumerator of $\mathcal{C}$. We present one application of this result to the partition function of the Potts model.

On the other hand, we show that there exist linear codes over any field different from rationals and $GF(p)$, $p$ prime, that are not triangular representable. We show that every construction of triangular representation fails on a very weak condition that a linear code and its triangular representation have to have the same dimension. 
\end{abstract}

%Keywords: linear code, triangular configuration, geometric representations, simplicial complex, weight enumerator

2000 MSC: 05C65, 94B05, 90C27, 55U10
% \begin{keyword}
%  linear code \sep triangular configuration \sep geometric representations \sep simplicial complex \sep weight enumerator
% \MSC 05C65 \sep 94B05 \sep 90C27 \sep 55U10
% \end{keyword}
%\end{frontmatter}

\section{Introduction}
\label{sec:intr}
The aim of this paper is to introduce a theory of geometric representations of general linear codes.
%It is a  continuation of Rytíř~\cite{rytir2} where binary codes were studied.
A seminal result of Galluccio and Loebl~\cite{loeblpfaffian} asserts that the Ising partition function on graph $G$ may be written as a linear combination of $4^{g(G)}$ Pfaffians, where $g(G)$ is the minimal genus of the closed Riemann surface in which $G$ can be embedded.
% A seminal result of Galluccio and Loebl~\cite{loeblpfaffian} asserts that the Ising partition function on graph $G$, that is the weight enumerator of the cut space of $G$, may be written as a linear combination of $4^{g(G)}$ Pfaffians, where $g(G)$ is the minimal genus of the closed Riemann surface in which $G$ can be embedded.
Recently, a topological interpretation of this result was given by Cimasoni and Reshetikhin~\cite{cimasoni-2007}.
We explain in Section~\ref{sec:mot} that the Ising partition function on graph $G$ may be described as the weight enumerator of the cycle space $\mathcal{C}$ of $G$.
Viewing the cycle space $\mathcal{C}$ as a linear code over $GF(2)$, a graph $G$ may be considered as a useful geometric representation of $\mathcal{C}$ which provides an important structure for the weight enumerator of $\mathcal{C}$, see Theorem~\ref{thm:galloebl}.
This motivated Martin Loebl to ask, about more than 10 years ago, the following question: Which binary linear codes are cycle spaces of simplicial complexes?
In general, for the linear codes with a geometric representation, one may hope to obtain a formula analogous to that of Theorem~\ref{thm:galloebl}.
This question remains open.
However, to extend the theory of Pfaffian orientations, it suffices to construct a geometric representation which carries over the weight enumerator only. This was achieved in Rytíř~\cite{rytir2} for binary linear codes. In this paper we present results for general linear codes.
%Potts
By another result of Galluccio and Loebl~\cite{loeblpfaffian2}, the q-Potts partition function of graph $G$ is determined by the row space of the oriented incidence matrix $O_G$ of graph $G$ over $GF(q)$. The row space of $O_G$ is a linear code, and so one surprising application of our results is a new formula for the q-Potts partition function, where q is a prime.
%Potts

We start with basic definitions.
A \dfn{linear code $\mathcal{C}$ of length $n$ and dimension $d$ over a field $\mathbb{F}$} is a linear subspace with dimension $d$ of the vector space $\mathbb{F}^n$. Each vector in $\mathcal{C}$ is called a \dfn{codeword}.
We define a partial order on $\mathcal{C}$ as follows:
Let $c=(c^1,\dots,c^n),d=(d^1,\dots,d^n)$ be codewords of $\mathcal{C}$. Then $c\preceq d$ if $c^i\neq 0$ implies $d^i\neq 0$ for all $i=1,\dots,n$. A codeword $d$ is \dfn{minimal} if $c\preceq d$ implies $c=d$ for all $c$.
The \dfn{weight} $w(c)$ of a codeword $c$ is the number of non-zero entries of $c$.
The \dfn{weight enumerator} of a finite code $\mathcal{C}$ is defined according to the formula
\begin{equation*}
W_\mathcal{C}(x):=\sum_{c\in \mathcal{C}} x^{w(c)}.
\end{equation*}
Let $\mathcal{C}\subseteq\mathbb{F}^n$ be a linear code over a field $\mathbb{F}$ and let $S$ be a subset of $\left\lbrace 1,\dots,n\right\rbrace $. \dfn{Puncturing} a code $\mathcal{C}$ along $S$ means deleting the entries indexed by the elements of $S$ from each codeword of $\mathcal{C}$. The resulting code is denoted by $\mathcal{C}/S$.

%The class of $1$-dimensional abstract simplicial complexes clearly coincides with the class of all graphs.
%A $2$-dimensional abstract simplicial complex we call \dfn{abstract triangular configuration}.
A \dfn{simplex} $X$ is the convex hull of an affine independent set $V$ in $\mathbb{R}^d$. The \dfn{dimension} of $X$ is $\left\lvert V\right\rvert-1$, denoted by $\dim X$.
The convex hull of any non-empty subset of $V$ that defines a simplex is called a \dfn{face} of the simplex.
A \dfn{simplicial complex} $\Delta$ is a set of simplices fulfilling the following conditions: Every face of a simplex from $\Delta$ belongs to $\Delta$ and the intersection of every two simplices of $\Delta$ is a face of both.
%The number of triangles in an (abstract) simplicial complex $\Delta$ is denoted by $\left\lvert\Delta\right\rvert$.
%A \dfn{subconfiguration} of a triangular configuration $\Delta$ is a triangular configuration $\Delta'$ such that $\Delta'\subseteq\Delta$.
%A \dfn{cycle} of a triangular configuration is a subconfiguration such that every edge is incident with an even number of triangles. A \dfn{circuit} is a minimal non-empty cycle under inclusion.
%\Delta_1^0\cup\Delta_2^0\cup\Delta_1^1\cup\Delta_2^1\cup(\Delta_1^2\cup\Delta_2^2)\setminus(\Delta_1^2\cap\Delta_2^2)
The dimension of $\Delta$ is $\max\left\lbrace\dim X\vert X\in\Delta\right\rbrace$.
Let $\Delta$ be a $d$-dimensional simplicial complex. We define the \dfn{incidence matrix} $A=\left(A_{ij}\right)$ as follows: The rows are indexed by $\left(d-1\right)$-dimensional simplices and the columns by $d$-dimensional simplices. We set
\begin{equation*}
A_{ij}:=\begin{cases}
         1& \text{if }(d-1)\text{-simplex }i\text{ belongs to }d\text{-simplex }j,\\
         0& \text{otherwise}.
        \end{cases}
\end{equation*}
This paper studies 2-dimensional simplicial complexes where each maximal simplex is a triangle or an edge. We call them \dfn{triangular configurations}.
The \dfn{cycle space} of $\Delta$ over a field $\mathbb{F}$, denoted $\ker\Delta$, is the kernel of the incidence matrix $A$ of $\Delta$ over $\mathbb{F}$, that is
$\{x\vert Ax=0\}$.
%, and $\mathcal{C}=\ker\Delta$ is said to be \dfn{represented} by $\Delta$.
%Let $\Delta$ be a triangular configuration. We give a numbering to triangles of $\Delta$.
%For a subconfiguration $C$ of $\Delta$, we let $\chi(C)=(\chi(C)^{t_1},\dots,\chi(C)^{t_{\lvert\Delta\rvert}})\in\left\lbrace 0,1\right\rbrace^{\left\lvert\Delta\right\rvert}$ denote its \dfn{incidence vector}, where $\chi(C)^t=1$ if $C$ contains the triangle $t$, and $\chi(C)^t=0$ otherwise.
%The entries of $\chi(C)$ are indexed by triangles which means that we give numbers to 
%It is well known that the kernel of $\Delta$ is the set of incidence vectors of cycles of $\Delta$.
%\begin{defn}
%\label{defn:puncturing}
%\end{defn}
A linear code $\mathcal{C}$ is \dfn{triangular representable} if there exists a triangular configuration $\Delta$ such that $\mathcal{C}=\ker\Delta/S$ for some set $S$ and there is a linear mapping between $\mathcal{C}$ and $\ker\Delta$ which is a bijection and maps minimal codewords to minimal codewords. For such $S$ we write $S=S(\ker\Delta,\mathcal{C})$.

\subsection{Motivation}
\label{sec:mot}
Our motivation to study geometric representations of linear codes lies in the theory of Pfaffian orientation in the study of the Ising problem on graphs. In this section we use the notation from Loebl and Masbaum~\cite{loeblarf}.
Let $G=(V(G),E(G))$ be a finite unoriented  graph. We say that $E'\subseteq E(G)$ is
\dfn{even} if the graph $(V(G),E')$ has even degree at each vertex.
%We say that $M\subset E(G)$ is a {\em perfect matching} if the
%graph $(V(G),M)$ has degree one at each vertex.  
Let $\mathcal{E}(G)$
denote the set of all even sets of edges of $G$.
The set of all even sets $\mathcal{E}(G)$ is called cycle space of graph $G$.
Note that, graph $G$ is 1-dimensional simplicial complex.
Let $A_G$ be its incidence matrix,
then the set of characteristic vectors of the even sets forms the kernel of $A_G$ over $GF(2)$.
%, and let $\mathcal{P}(G)$ denote the set of all perfect matchings of $G$. 

We assume that a variable $x_e$ is associated with each edge
$e$, and define the generating polynomial for even sets,  $\E_G$ , in $
\BZ[(x_e)_{e\in E(G)}]$,  as follows: 
$$\E_G(x)= \sum_{E'\in\mathcal{E}(G) }\,\,\prod_{e\in E'}x_e~.$$

Let $\mathcal{C}$ be the kernel of the incidence matrix $A_G$ of graph $G$. Then there is the following relation between the weight enumerator of $\mathcal{C}$ and the generating polynomial of even sets of $G$
$$
W_{\mathcal{C}}(x)=\E_G(z)\Big|
  _{\textstyle{z_e:=x \ \forall  e\in
  E(G)}}.
$$

Next, we describe the equivalence of Ising partition function and the generating function of even subgraphs.
The Ising partition function on graph $G$ is defined by 
$$
Z_G^{\mathrm{Ising}}(\beta) = Z_G^{\mathrm{Ising}}(x)\Big|
  _{\textstyle{x_e:=e^{\beta J_e} \ \forall  e\in
  E(G)}}
$$
 where the $J_e$ $( e\in E(G))$ are weights (coupling constants)
 associated with the edges of the graph $G$, the parameter $\beta$ is the inverse
 temperature,  and
$$ 
Z_G^{\mathrm{Ising}}(x)= \sum_{\sigma:V(G)\rightarrow \{1,-1\}} \  \prod_{e= \{u,v\}\in E(G)}x_e^{\sigma(u)\sigma(v)}.
$$
The theorem of van der Waerden \cite{vanderwarder} (see \cite[Section 6.3]{loeblbook}
for a proof) states that $Z_G^{\mathrm{Ising}}(x)$ is the same as $\E_G(x)$ up to
change of variables and 
multiplication by a constant factor: 

$$
Z_G^{\mathrm{Ising}}(x)= 2^{|V(G)|}\left(\prod_{e\in E(G)}\frac{x_e+ x_e^{-1}}{2}\right) \E_G(z)\Big|
  _{\textstyle{z_e:= \frac{x_e- x_e^{-1}}{x_e+ x_e^{-1}}}}.
$$

The significance of geometric properties of graph $G$ in studying the Ising partition function $Z_G^{\mathrm{Ising}}(x)$ and the generating function of the even subsets of edges of $G$, $\E_G(x)$, is expressed by the following theorem.

\begin{thm}[Galluccio and Loebl \cite{loeblpfaffian}]
\label{thm:galloebl}
If $G$
embeds into an orientable surface of genus $g$, then the
even subgraph
polynomial $\E_G(x)$ can be
expressed as a linear combination of $4^g$ Pfaffians of matrices constructed from the embedding of $G$.
\end{thm}

\subsection{Main results}
%The main results are:

\begin{thm}
\label{thm:repr1}
Let $\mathcal{C}$ be a linear code over a field $\mathbb{F}$, $\mathcal{C}\subseteq\mathbb{F}^n$, and let $B$ be a basis of $\mathcal{C}$ such that for each $b\in B$, each entry of $b$ belongs to a cyclic subgroup $G(b)$ of the additive group of the field $\mathbb{F}$. Then $\mathcal{C}$ is triangular representable. Moreover, if $\mathcal{C}$ is finite, then there exists a triangular representation $\Delta$
% and a positive integer $e$ linear in $n$
such that: if $\sum_{i=0}^{m}a_ix^i$ is the weight enumerator of $\ker\Delta$ then
\begin{equation*}
W_\mathcal{C}(x)=\sum_{i=0}^{m}a_ix^{(i\bmod e)/2},
\end{equation*}
where $e=(\lvert S(\ker\Delta,\mathcal{C})\rvert-n)/\dim\mathcal{C}$.
\end{thm}
\begin{cor}
\label{thm:repr2}
The conclusion of Theorem~\ref{thm:repr1} holds for the linear codes over rationals and over $GF(p)$, where $p$ is a prime.
%Let $\mathcal{C}$ be a linear code  over rationals or over $GF(p)$, where $p$ is a prime. Then $\mathcal{C}$ is triangular representable. Moreover, if $\mathcal{C}$ is finite, then there exists a triangular representation $\Delta$ and a positive integer $e$ linear in length of $C$ such that if $\sum_{i=0}^{m}a_ix^i$ is the weight enumerator of $\Delta$ then
% \begin{equation*}
% W_\mathcal{C}(x)=\sum_{i=0}^{m}a_ix^{(i\bmod e)/2}.
% \end{equation*}
\end{cor}

Corollary~\ref{thm:repr2} was proved for $GF(2)$ in Rytíř~\cite{rytir2}.
On the other hand, we have the following negative results for the remaining fields.

\begin{thm}
\label{thm:nonrepr}
Let $\mathcal{C}$ be a linear code over a field $\mathbb{F}$ such that every basis $B$ of $\mathcal{C}$ contains a vector $b$ so that its entries $f,p$ do not belong to the same cyclic subgroup of the additive group of $\mathbb{F}$. Then code $\mathcal{C}$ is not triangular representable.
\end{thm}

\begin{cor}
\label{cor:nonrepr}
Let $\mathbb{F}$ be a field different from rationals and $GF(p)$, where $p$ is a prime. Then there exists a linear code over $\mathbb{F}$ that is not triangular representable.
\end{cor}

In Rytíř~\cite{rytir2}, for every triangular configuration $\Delta$, we constructed a triangular configuration $\Delta'$ such that there exists a bijection between the cycle space over $GF(2)$ of $\Delta$ and the set of the perfect matchings of $\Delta'$. We conjecture that this result can be extended to other finite fields.
Finally, see Section~\ref{sec:multi} for the multivariate versions of the above theorems.

%Potts
\subsection{Potts model}
We give an application (suggested by Martin Loebl) of our results to the Potts partition function. The \dfn{q-Potts partition function} of a graph $G$ is defined according to the formula
$$P^q(G,x)=\sum_{s:V\mapsto \{0,\dots,q-1\}}x^{H(s)}.$$
The Hamiltonian $H(s)$ is defined as $\sum_{uv\in E} w(uv)\delta(s(u),s(v))$ where $\delta$ is Kronecker delta and $w(uv)$ is the weight of edge $uv$.
The Potts partition function is one of the extensively studied functions in statistical physics; see the book by Loebl~\cite{loeblbook}. We note that 2-Potts partition function is equivalent to the Ising partition function and that $P^q(G,x)$ is also called the multivariate Tutte polynomial of $G$; see the survey by Sokal~\cite{sokal}. 
\begin{thm}
\label{cor:potts}
Let $G$ be a connected graph such that all edges of $G$ have the same weight $w$. Let $q$ be a prime. Then there exists triangular configuration $\Delta$ such that if $\sum_{i=0}^{m}a_ix^i$ is the weight enumerator of $\ker\Delta$ then
\begin{equation*}
P^q(G,x)=\frac{q}{\prod_{uv\in E} x^{-w}}\sum_{i=0}^{m}a_ix^{-w((i\bmod e)/2)},
\end{equation*}
where $e$ is a positive integer linear in the number of edges of $G$.
\end{thm}
The case of different weights is treated in Section~\ref{sec:multi}.
% The interesting question is whether triangular representations of cut spaces of graphs have some special properties.

\section{Triangular representations}

In this section all computations are done over a fixed field $\mathbb{F}$.
Let $\Delta_1$, $\Delta_2$ be subconfigurations of a triangular configuration $\Delta$.
%We denote the subset of $d$-dimensional simplices of $\Delta$ by $\Delta^d$.
We denote the set of vertices of $\Delta$ by $V(\Delta)$, the set of edges by $E(\Delta)$ and the set of triangles by $T(\Delta)$.
The \dfn{difference} of $\Delta_1$ and $\Delta_2$, denoted by $\Delta_1-\Delta_2$, is defined to be the triangular configuration obtained from $V(\Delta_1)\cup E(\Delta_1)\cup T(\Delta_1)\setminus T(\Delta_2)$ by removing the edges and vertices that are not contained in any triangle in $T(\Delta_1)\setminus T(\Delta_2)$.
%The \dfn{symmetric difference} of $\Delta_1$ and $\Delta_2$, denoted by $\Delta_1\bigtriangleup\Delta_2$, is defined to be $\Delta_1\bigtriangleup\Delta_2:=\left(\Delta_1\cup\Delta_2\right)-\left(\Delta_1\cap\Delta_2\right)$.
An \dfn{abstract simplicial complex} on a finite set $V$ is a family $\Delta$ of subsets of $V$ closed under taking subsets.
Let $X$ be an element of $\Delta$. The \dfn{dimension} of $X$ is $\left\lvert X\right\rvert-1$, denoted by $\dim X$.
The dimension of $\Delta$ is $\max\left\lbrace\dim X \vert X\in\Delta\right\rbrace$.
Every simplicial complex defines an abstract simplicial complex on the set of vertices $V$, namely the family of sets of vertices of simplexes of $\Delta$. We denote this abstract simplicial complex by $\mathcal{A}(\Delta)$.
The \dfn{geometric realization} of an abstract simplicial complex $\Delta$ is a simplicial complex $\Delta'$ such that $\Delta=\mathcal{A}(\Delta')$.
It is well known that every finite $d$-dimensional abstract simplicial complex can be realized as a simplicial complex in $\mathbb{R}^{2d+1}$. We choose a geometric realization of an abstract simplicial complex $\Delta$ and denote it by $\mathcal{G}(\Delta)$.
Let $\Delta_1,\Delta_2$ be triangular configurations. The union of $\Delta_1,\Delta_2$ is defined to be $\Delta_1\cup\Delta_2:=\mathcal{G}(\mathcal{A}(\Delta_1)\cup\mathcal{A}(\Delta_2))$.
We start with construction of basic building blocks.
\subsection{Triangular configuration $B^n$}
\label{sec:bn}
The triangular configuration $B^n$ consists of $n$ disjoint triangles as is depicted in Figure~\ref{fig:trianconf B}. We denote the triangles of $B^n$ by $B^n_1,\dots,B_n^n$.

\begin{figure}[h]
\begin{center}
 \includegraphics[width=150pt]{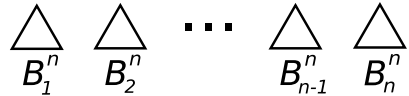}
\end{center}
\caption{Triangular configuration $B^n$.}
\label{fig:trianconf B}
\end{figure}

\subsection{Orientation}
\label{sec:orientation}
An \dfn{oriented triangular configuration} $\vec{\Delta}$ is a triangular configuration $\Delta$ with an assignment $o:T(\Delta)\mapsto\{+,-\}$ of signs to triangles.
We denote the subset of triangles $\{t\in T(\Delta)\vert o(t)=+\}$ by $\left[\vec{\Delta}\right]_+$ and $\{t\in T(\Delta)\vert o(t)=-\}$ by $\left[\vec{\Delta}\right]_-$.

\subsection{Oriented tunnel $\vec{T}$}
An \dfn{empty triangle} is a set of three edges forming a boundary of a triangle.

The \dfn{oriented tunnel} $\vec{T}$ is defined by Figure~\ref{fig:tunnel}.
\begin{figure}[h]
	\begin{center}
	\includegraphics[width=320pt]{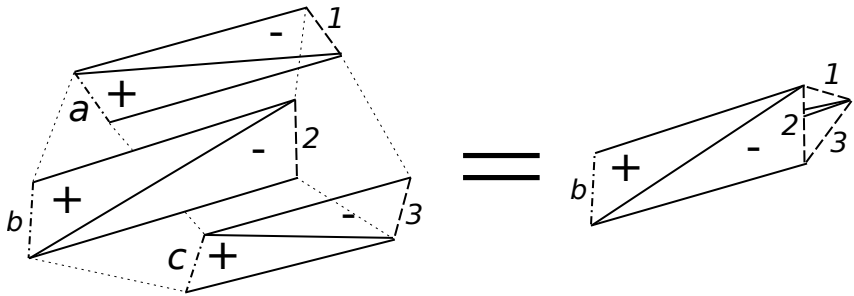}
	\end{center}
\caption{Oriented triangular tunnel $\vec{T}$.}
\label{fig:tunnel}
\end{figure}
The oriented tunnel has two empty triangles $\left\lbrace a,b,c\right\rbrace$ and $\left\lbrace 1,2,3\right\rbrace$.
The empty triangle $\left\lbrace a,b,c\right\rbrace$ is called \dfn{positive end} of the tunnel, and the empty triangle $\left\lbrace 1,2,3\right\rbrace$ is called \dfn{negative end}.

A \dfn{tunnel} is obtained from oriented tunnel by deleting the signs.
%In particular, triangles $\left\lbrace 1,2,3\right\rbrace $ and $\left\lbrace a,b,c\right\rbrace $ are not elements of $T$.
%We denote the set of '$+$' triangles and '$-$' triangles of $\vec{T}$ by $\left[\vec{T}\right]_{+}$ and $\left[\vec{T}\right]_{-}$, respectively.

%\subsection{Union of triangular configurations}
\begin{prop}
\label{prop:tunnel}
Let $T$ be a tunnel and let $v=\left(v^t\right)_{t\in T(T)}$ be a vector whose entries are indexed by triangles of $T$ satisfying: For each inner edge $e$,
% the sum of $v^t$ in a field $\mathbb{F}$ over all triangles incident with $e$ is $0$ i.e., 
$\sum_{e\in t}v^t=0$. Then 
$v^{t_a}=v^{t_b}=v^{t_c}$ and $v^{t_1}=v^{t_2}=v^{t_3}$,
where for $f\in\{a,b,c,1,2,3\}$, $t_f$ denotes the triangle of $T$ containing the edge $f$.
Moreover, $v^{t_a}=-v^{t_1}$. \qed
\end{prop}

\subsection{Linking triangles by oriented tunnel}

Let $\vec{\Delta}$ be an oriented triangular configuration. Let $t_1$ and $t_2$ be two edge disjoint triangles of $\vec{\Delta}$.

The \dfn{link} from $t_1$ to $t_2$ in $\vec{\Delta}$ is the oriented triangular configuration $\vec{\Delta'}$ defined as follows.
Let $\vec{T}$ be an oriented triangular tunnel as in Figure~\ref{fig:tunnel}. Let $t_1^1,t_1^2,t_1^3$ and $t_2^1,t_2^2,t_2^3$ be edges of $t_1$ and $t_2$, respectively.
We relabel edges of $\vec{T}$ such that $\left\{a,b,c\right\}=\left\{t_1^1,t_1^2,t_1^3\right\}$ and $\left\{1,2,3\right\}=\left\{t_2^1,t_2^2,t_2^3\right\}$. We let $\vec{\Delta'}:=\vec{\Delta}\cup \vec{T}$.

Note that, the link has a direction, the triangle $t_1$ is incident with the positive end of $\vec{T}$ in $\vec{\Delta'}$, and $t_2$ is incident with the negative end of $\vec{T}$ in $\vec{\Delta'}$.

Analogously, we define linking from $t_1$ to $t_2$ of edge disjoint $t_1,t_2$ where at least one of $t_1,t_2$ is an empty triangle of $\vec{\Delta}$.

Further, if $\Delta$ is (non-oriented) triangular configuration and $t_1,t_2$ are edge disjoint triangles, then link from $t_1$ to $t_2$ is defined as above, but replacing oriented tunnel by tunnel.

\subsection{Triangular configuration $\Delta^\mathcal{C}_B$}
\label{sec:representation}

Let $\mathcal{C}$ be a linear code of length $n$ and dimension $d$ over a field $\mathbb{F}$.
Let $B=\left\{b_1,\dots,b_d\right\}$ be a basis of $\mathcal{C}$.
We say that a basis $B$ is \dfn{representable} if for each $b\in B$ each entry of $b$ belongs to a cyclic subgroup $G(b)$ of the additive group of the field $\mathbb{F}$.
In this section we suppose that the linear code $\mathcal{C}$ has a representable basis $B$.
This is used in the key Proposition~\ref{prop:multispherespace} bellow.
We construct the triangular configuration $\Delta^C_B$ as follows.
First, for every basis vector $b$ we construct a triangular configuration $\Delta^C_{b}$.
Then, the triangular configuration $\Delta^\mathcal{C}_B$ is the union of $\Delta^\mathcal{C}_{b}$, $b\in B$.
\subsubsection{Triangular configuration $\Delta^\mathcal{C}_b$}

\begin{prop}
\label{prop:multispherespace}
Let $\mathbb{F}$ be a field. Let $a_1,a_2,\dots,a_k$ be a subset of distinct non-zero elements of a cyclic subgroup $G$ of the additive group of $\mathbb{F}$. Let $n$ be a positive integer. Then there exists a triangular configuration $\mathcal{M}$ such that $\dim\ker \mathcal{M} = 1$ and there exists a vector $v\in\ker\mathcal{M}$ having $a_1,a_2,\dots,a_k$ among its entries. Moreover, the vector $v$ contains each entry $a_i$ at least $n$ times and $v$ is non-zero on all entries.
\end{prop}
The proof is postponed to Section~\ref{sec:proofprop}. We recall that $b$ is a basis vector of a representable basis $B$ and that the length of $\mathcal{C}$ is $n$. Let $b=(b^1,b^2,\dots,b^n)$ and let $a_1,a_2,\dots,a_k$ be all different entries of $b$.
Let $\mathcal{M}(b)$ and $v(b)=\left(v^t\right)_{t\in T(\mathcal{M}(b))}(b)$ be a triangular configuration and a vector provided by Proposition~\ref{prop:multispherespace}. Next, we construct $\Delta^\mathcal{C}_b$.
We proceed in three steps.
In the first step, the construction starts 
%Let $a_1,a_2,\dots,a_k$ be all different entries of $b$ and let $n$ be the length of $\mathcal{C}$.
%Given $a_1,a_2,\dots,a_k$ and $n$, Proposition~\ref{prop:multispherespace} gives us triangular configuration $\mathcal{M}(b)$ and vector $v(b)=\left(v^t\right)_{t\in T(\mathcal{M}(b))}(b)$.
%Let $\mathcal{M}$ be the triangular configuration and let $v$ be the vector from the above proposition applied on elements $a_1,a_2,\dots,a_k$ of the cyclic subgroup $G(b)$ and on integer $n$, that is the length of $\mathcal{C}$.
%The triangular configuration $\Delta^\mathcal{C}_{b}$ is obtained
from $B^n\cup \mathcal{M}(b)$. See Section~\ref{sec:bn} for definition of $B^n$.

% We note that $\Delta^\mathcal{C}_{b_i}$ is not oriented triangular configuration.
% $m$ is even and $m\geq n$, $m\geq 4$ and 
In the second step, we make the following tunnels.
Let $J(b)$ be the set of indices of non-zero entries of $b$.
Let $g$ be an injection $g:J(b)\mapsto T(\mathcal{M}(b))$ such that $\forall j\in J(b), v(b)^{g(j)}=b^j$. We note that $g$ exists since the multiplicity of each $a_i$ in $b$ is at most $n$. For each $j\in J(b)$ we create link by tunnel from the triangle $g(j)$ to the triangle $B^n_j$.
%For each $j\in J(b)$
%Let $b_i^j$ equals $a$. 
%let $t$ be a triangle of $\mathcal{M}$ such that $v^t=b^j$. We create link by tunnel from the triangle $t$ to the triangle $B^n_j$.

In the third step, we remove the triangles $g(j)$, $j\in J(b)$, from $\Delta^\mathcal{C}_{b}$.
%Finally, we remove the triangles of $B^n$ that are not joined with $\mathcal{M}$.
An example of $\Delta^\mathcal{C}_{b}$ for $b=\left(b^1,0,\dots,b^{n-1},0\right)$ is depicted in Figure~\ref{fig:trianrep cycle}.
%We denote 
%Thus, the triangular configuration $\Delta^\mathcal{C}_{b_i}$ contains $B^n_j$ if and only if $j\in J^i$.
%We note that

\begin{figure}[hp]
\begin{center}
 \includegraphics[width=150pt]{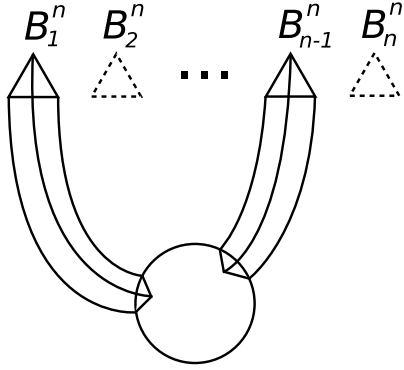}
\end{center}
\caption{$\Delta^\mathcal{C}_{b}$ represents a basis vector $\left(b^1,0,\dots,b^{n-1},0\right)$ of $\mathcal{C}$.}
\label{fig:trianrep cycle}
\end{figure}

\begin{prop}
\label{prop:gen}
%Denote $\left\lvert \Delta^\mathcal{C}_{b} \right\rvert$ by $m(b)$.
The dimension of $\ker\Delta^\mathcal{C}_{b}$ is equal to $1$. Moreover,
there exists a generator $u(b)$ of $\ker\Delta^\mathcal{C}_{b}$ such that $u(b)^{B^n_j}=b^j$ for $j=1,\dots,n$ and $u(b)^t\neq 0$ for all $t\in T(\Delta^\mathcal{C}_{b}- B^n)$.

%We denote such vector $u$ by $\gen{\Delta^\mathcal{C}_{b}}$.
\end{prop}
\begin{proof}
By Proposition~\ref{prop:multispherespace}, the dimension of $\ker\mathcal{M}(b)$ is $1$. The triangular configuration $\Delta^\mathcal{C}_b$ is obtained from $\mathcal{M}(b)\cup B^n$ by linking some disjoint triangles of $B^n$ by tunnel to some triangles of $\mathcal{M}(b)$ and by removing triangles of $\mathcal{M}(b)$ that are linked by a tunnel. Hence, the dimension of $\ker\Delta^\mathcal{C}_b$ is $1$.

Let $u$ be a vector from $\ker\mathcal{M}(b)$ given by Proposition~\ref{prop:multispherespace}. If entry $b^j$ is non-zero, in construction of $\mathcal{M}(b)$ we link the triangle $B^n_j$ with a triangle $t$ of $T(\mathcal{M}(b))$ such that $u^t=b_j$ and then we remove triangle $t$. Therefore by Proposition~\ref{prop:tunnel}, we can extend vector $u$ to $u(b)$ such that $u(b)^{B^n_j}=b^j$. If entry $b^k$ is zero, then triangle $B^n_j$ is isolated. Therefore, $u(b)^{B^n_j}=0$.

The vector $u$ given by Proposition~\ref{prop:multispherespace}, is non-zero on all entries. By Proposition~\ref{prop:tunnel}, all entries indexed by triangles of tunnels linked to $\mathcal{M}(b)$ are non-zero. Hence, $u(b)^t\neq 0$ for all $t\in T(\Delta^\mathcal{C}_{b}- B^n)$.
\end{proof}

\subsubsection{Construction of $\Delta^\mathcal{C}_B$}
\label{sec:constdeltaCB}

Triangular configurations $\Delta^\mathcal{C}_{b}$, $b\in B$, share only triangles of $B^n$. Hence, $\mathcal{A}(\Delta^\mathcal{C}_{b})\cap\mathcal{A}(\Delta^\mathcal{C}_{b'})=\mathcal{A}(B_n)$ holds for $b\neq b'$; $b,b'\in B$.

The triangular configuration $\Delta^\mathcal{C}_B$ is the union of $\Delta^\mathcal{C}_{b}$, $b\in B$.
An example of a triangular configuration $\Delta^\mathcal{C}_B$ is depicted in Figure~\ref{fig:trianrep}.

%with respect to basis $(1,0,\dots,1,0), (1,0,\dots,1,0),\dots,(0,0,\dots,0,1)$.
%Let $\mathcal{C}$ be a binary linear code of dimension $d$.

%As $S_m$ exists for every even $k\geq4$, balanced representation always exists.

We define the following vectors of $\ker\Delta^\mathcal{C}_B$. The vector $\gen{\Delta^\mathcal{C}_{b}}$ is defined as $\gen{\Delta^\mathcal{C}_{b}}^t:=u(b)^t$ for $t\in T(\Delta^\mathcal{C}_{b})$ and $\gen{\Delta^\mathcal{C}_{b}}^t:=0$ for $t\in T(\Delta^\mathcal{C}_B-\Delta^\mathcal{C}_{b})$.
As a corollary of Proposition~\ref{prop:gen} we get:
\begin{cor}
\label{cor:genentries}
Each entry of vector $\gen{\Delta^\mathcal{C}_{b}}$ is non-zero on each entry indexed by a triangle of $\Delta^\mathcal{C}_{b}- B^n$.
Moreover, the vectors $\gen{\Delta^\mathcal{C}_{b}}$, $b\in B$, are linearly independent.\qed
\end{cor}

\begin{definition}
\label{dfn:mappingf}
%We denote the addition in $\mathbb{F}$ by $+^{\mathbb{F}}$ or $\sideset{^{\mathbb{F}}}{}\sum$.
%We denote by $\ker\Delta^\mathcal{C}_B$ the cycle space of the triangular configuration
%$\Delta^\mathcal{C}_B$.
%Let $C$ be a binary code. Let $ $ be its triangular repre
We define a linear mapping $f\colon\mathcal{C}\mapsto \ker\Delta^\mathcal{C}_B$ in the following way:
Let $c$ be a codeword of $\mathcal{C}$ and let $c=\sum_{b\in B}\alpha_bb$ be the unique expression of $c$, where $\alpha_b\in\mathbb{F}$.
%We give a numbering ot 
We define $f(c):=\sum_{b\in B}\alpha_b\gen{\Delta^\mathcal{C}_{b}}$.
%Let us remind of the incidence vector definition.
The entries of $f(c)$ are indexed by the triangles of $\Delta^\mathcal{C}_B$.
%We have $f(c)^{B^n_j}=1$ if and only if $\bigtriangleup_{i\in I} \Delta^\mathcal{C}_{b_i}$ contains the triangle $B^n_j$.
Let $R=\{1,\dots,n\}$ be the set of coordinates of $\mathcal{C}$. We define an injection $\mu:R\mapsto T(\Delta)$ as: $\mu(i)=B^n_i$ for $i=1,\dots,n$.
\end{definition}

\begin{figure}[hp]
\begin{center}
 \includegraphics[width=150pt]{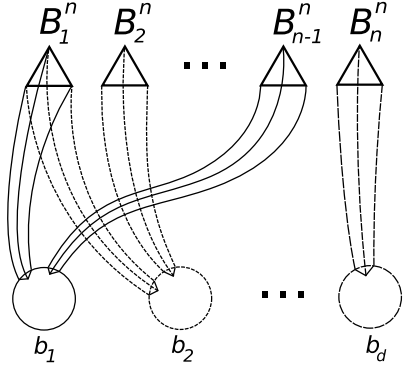}
 % minors3.eps: 1179666x1179666 pixel, 0dpi, infxinf cm, bb=
\end{center}
\caption{An example of a triangular configuration $\Delta^\mathcal{C}_B$, where $B=\{b_1,\dots,b_d\}$.}
\label{fig:trianrep}
\end{figure}

\begin{prop}
\label{prop:mappingf}
Denote $\left\lvert T\left(\Delta^\mathcal{C}_{B}\right) \right\rvert$ by $m$.
Let $c=(c^1,\dots,c^n)$ and
\begin{equation*}
f(c)=\left(f(c)^{B^n_1},\dots,f(c)^{B^n_n},f(c)^{n+1},\dots,f(c)^{m}\right). 
\end{equation*}
Then $f(c)^{B^n_j}=c^j$ for all $j=1,\dots,n$ and all $c\in\mathcal{C}$.
\end{prop}
\begin{proof}
$$f(c)^{B^n_j}=\sum_{b\in B}\alpha_b \gen{\Delta^\mathcal{C}_{b}}^{B^n_j}.$$
By Proposition~\ref{prop:gen} and by definition of $\gen{\Delta^\mathcal{C}_{b}}$,
$$\sum_{b\in B}\alpha_b \gen{\Delta^\mathcal{C}_{b}}^{B^n_j}=\sum_{b\in B}\alpha_b b^{B^n_j}=c^j.$$

% We show the proposition by the induction on the degree $\vd{c}$ of $c$.
% The codeword $c$ is equal to $\sideset{^{\mathbb{F}}}{}\sum_{i=1}^d\alpha_i b_i$.
% If $\vd{c}=0$, then $c=0$ and $f(c)=0$. Hence, the proposition holds for vectors of degree $0$.
% % Thus, $f(c)$ is the incidence vector of the empty triangular configuration.
% If $\vd{c}$ is greater than $0$, then there exists $k$ such that $\alpha_k\neq 0$.
% The codeword $c+^{\mathbb{F}}(-\alpha_k)b_k$ has a degree less than $c$. By the induction assumption, the proposition holds for $c+^{\mathbb{F}}(-\alpha_k)b_k$.
% Let $b_k=(b_k^1,\dots,b_k^n)$.
% From the definition of $\Delta^\mathcal{C}_{b_k}$, the equality $b_k^j=\gen{\Delta^\mathcal{C}_{b_k}}^{B^n_j}$ holds for all $j=1,\dots,n$.
% Therefore,
% \begin{equation*}
% c^j=(c^j+^{\mathbb{F}}(-\alpha_k)b_k^j)+^{\mathbb{F}}\alpha_kb_k^j=\sideset{^{\mathbb{F}}}{}\sum_{i=1,i\neq k}^d\alpha_i\gen{\Delta^\mathcal{C}_{b_i}}^{B^n_j}+^{\mathbb{F}}\alpha_k\gen{\Delta^\mathcal{C}_{b_k}}^{B^n_j}=f(c)^{B^n_j}
% \end{equation*}
% % 
% % \chi(\bigtriangleup_{i\in I\setminus\{k\}}\Delta^\mathcal{C}_{b_i})^{B^n_j}+^2\chi(\Delta^\mathcal{C}_{b_k})^{B^n_j}=f(c)^{B^n_j}
% for all $j=1,\dots,n$. 
\end{proof}

\begin{cor}
\label{cor:inj}
The linear mapping $f$ is injective.\qed
\end{cor}

\begin{lem}
\label{lem:cycletriangles}
For each non-zero vector $w$ of $\ker\Delta^\mathcal{C}_B$ there exists $b\in B$ and $\gamma_b\neq 0$ such that $w^t=\gamma_b\gen{\Delta^\mathcal{C}_b}^t$ for all $t\in T(\Delta^\mathcal{C}_b- B^n)$.
%Every non-zero vector of $\ker\Delta^\mathcal{C}_B$ has non-zero entries on coordinates indexed by triangles of $\Delta^\mathcal{C}_{b}-B^n$ for some $b\in B$.
\end{lem}
\begin{proof}
The kernel $\ker B^n$ is $\emptyset$, since the triangles of $B^n$ are disjoint.
Therefore, every non-zero vector $w\in\ker\Delta^\mathcal{C}$ has a non-zero element $w^t\neq0$ indexed by a triangle $t\in T(\Delta^\mathcal{C}_{b}- B^n)$ for some $b\in B$.

Let $j$ be an index such that $b^j\neq0$. Let $t_1,t_2,t_3$ be three triangles of $\Delta^\mathcal{C}_b$ touching edges of $B^n_j$. Then by proposition~\ref{prop:tunnel}, $w^{t_1}=w^{t_2}=w^{t_3}$. By Proposition~\ref{prop:gen}, the dimension of $\Delta^\mathcal{C}_b$ is $1$. Hence, there is a non-zero scalar $\gamma_b$ such that $w^t=\gamma_b\gen{\Delta^\mathcal{C}_b}^t$ for all $t\in T(\Delta^\mathcal{C}_b- B^n)$.

\end{proof}

\begin{thm}
\label{thm:trianrep}
Let $\mathcal{C}$ be a linear code with a representable basis $B$ and let $\Delta^\mathcal{C}_B$ be the triangular configuration from this section.
Then $\mathcal{C}=\ker\Delta^\mathcal{C}_B/S$, where $S$ is a set of indices. Moreover, the linear mapping $f$ defined above is a bijection between the linear code $\mathcal{C}$ and $\ker\Delta^\mathcal{C}_B$ which maps minimal codewords to
minimal codewords.
%where $S$ is the set triangles of %$\Delta^\mathcal{C}_B$ out of $B^n$.
%Moreover, a bijection between $\mathcal{C}$ and $\ker\Delta^\mathcal{C}_B$ exists, that %maps minimal codewords to minimal codewords, and %$\dim\mathcal{C}=\dim\ker\Delta^\mathcal{C}_B$.
\end{thm}
\begin{proof}
By Proposition~\ref{prop:mappingf}, the code $\mathcal{C}$ equals $\ker\Delta^\mathcal{C}_B/S$, where $S$ is the set of triangles of $\Delta^\mathcal{C}_B- B^n$.

By Corollary~\ref{cor:inj}, the mapping $f$ is injective.
%Codes $\mathcal{C}$ and $(\ker\Delta^\mathcal{C}_B)/S$ have the same length. So we %identify index $i$ with index $B^n_i$. By Proposition~\ref{prop:mappingf}, $f(c)^i=c^i$ %for all $i=1,\dots,n$. Therefore, $c=f(c)/S$, where $f(c)/S$ means deleting the entries %indexed by the elements of $S$ from $f(c)$.
%Thus, $\mathcal{C}\subseteq(\ker\Delta^\mathcal{C}_B)/S$.
It remains to be proven that $\dim\mathcal{C}=\dim\ker\Delta^\mathcal{C}_B$.
 %$\ker\Delta^\mathcal{C}_B$ contains $\left\{f(b_1),\dots,f(b_d)\right\}$, 
%$\dim\mathcal{C}\leq \dim\ker\Delta^\mathcal{C}_B$.
We show that every codeword $w$ of $\ker\Delta^\mathcal{C}_B$ is in the linear span of $\left\{f(b)\vert b\in B\right\}$.
Let $B(w)$ be the following set of basis vectors $\{b\in B\vert\exists t\in T(\Delta^\mathcal{C}_{b}- B^n) \text{ such that } w^t\neq 0\}$. By Lemma~\ref{lem:cycletriangles}, the set $B(w)$ is not empty. By Lemma~\ref{lem:cycletriangles}, vector $w-\sum_{b\in B(w)}\gamma_b\gen{\Delta^\mathcal{C}_b}$ is non-zero only on coordinates indexed by triangles of $B^n$. Since $\ker B^n=\emptyset$, vector $w-\sum_{b\in B(w)}\gamma_b\gen{\Delta^\mathcal{C}_b}$ is $0$. Hence, the vector $w$ is in the span of $\left\{f(b_1),\dots,f(b_d)\right\}$.

% Suppose on the contrary that some codeword of $\ker\Delta^\mathcal{C}_B$ is not in the span of $\left\{f(b_1),\dots,f(b_d)\right\}$.
% Let $c$ be such a codeword with the minimal possible weight $w(c)$.
% %Let $K$ be a cycle of $\Delta^\mathcal{C}_B$ such that $\chi(K)=c$.
% By Lemma~\ref{lem:cycletriangles}, the codeword $c$ has non-zero elements indexed by $\Delta^\mathcal{C}_{b_i}- B^n$ for some $i\in \left\{1,\dots,d\right\}$.
% Let $t$ be a triangle of $\Delta^\mathcal{C}_{b_i}- B^n$.
% Let $\alpha$ be equal to $c^t/\gen{\Delta^\mathcal{C}_{b_i}}^t$.
% Since $\left\lvert\Delta^\mathcal{C}_{b_i}- B^n\right\rvert>\left\lvert B^n\right\rvert$ and by Lemma~\ref{lem:multispheretrianglesdetermined}, the inequality
% $w(c-\alpha\gen{\Delta^\mathcal{C}_{b_i}}) < w(c)$ holds.
% This is a contradiction.

%Since $\left\lvert %\mathcal{C}\right\rvert=\left\lvert\ker\Delta^\mathcal{C}_B\right\rvert$, mapping $f$ is a %bijection, so $\mathcal{C}=(\ker\Delta^\mathcal{C}_B)/S$.

Finally, we show that $f$ maps minimal codewords to minimal codewords.
Recall a partial order on $\mathcal{C}$.
Let $r=(r^1,\dots,r^n),s=(s^1,\dots,s^n)$ be codewords of $\mathcal{C}$. Then $r\preceq s$ if $r^i\neq 0$ implies $s^i\neq 0$ for all $i=1,\dots,n$. A codeword $s$ is \dfn{minimal} if $r\preceq s$ implies $r=s$ for all $r$.
Let $s$ be a minimal codeword.
Suppose on the contrary that $f(s)$ is not a minimal codeword of $\ker\Delta^\mathcal{C}_B$.
Then $f(r)\prec f(s)$ for some codeword $r$. Therefore, for all $j=1,\dots,n$; $f(r)^{B^n_j}\neq 0$ implies that $f(s)^{B^n_j}\neq 0$. By Proposition~\ref{prop:mappingf}, $r^j=f(r)^{B^n_j}$ and $s^j=f(s)^{B^n_j}$, for all $j=1,\dots,n$. Hence, $r^j\neq 0$ implies that $s^j\neq 0$, for all $j=1,\dots,n$. Thus, $r\prec s$. This is a contradiction.
%So, codeword $f(c)$ distinct of $f(d)$ exists such that $f(c)\preceq f(d)$.
%$c^i=f(c)^i=1$ implies $d^i=f(d)^i=1$. Hence, $c\preceq d$ and $c$ is distinct of $d$. This is a contradiction with minimality of $d$. \qed
\end{proof}

\subsection{Balanced triangular configuration $\Delta^\mathcal{C}_B$}
\label{sec:balrepresentation}

A triangular configuration $\Delta^\mathcal{C}_B$ is \dfn{balanced} if there is an integer $e$ such that $\left\lvert T\left(\Delta^\mathcal{C}_{b}\right)\right\rvert-w(b)=e$ for all $b\in B$.
This $e$ is denoted by $\ex{\Delta^\mathcal{C}_B}$. A code $\mathcal{C}$ is \dfn{even} if all codewords have an even weight. The following proposition is straightforward.
\begin{prop}
\label{prop:expcont}
$\ex{\Delta^\mathcal{C}_B}=\left\lvert T\left(\Delta^\mathcal{C}_B- B^n\right)\right\rvert/\dim\mathcal{C}$\qed
\end{prop}

\begin{prop}
\label{prop:balancedrepr}
Let $C$ be an even linear code with a representable basis $B=\{b_1,\dots,b_d\}$. Let $n$ be an integer. Then there exists a balanced triangular configuration $\Delta^\mathcal{C}_B$ such that $n<\ex{\Delta^\mathcal{C}_B}$.
\end{prop}
\begin{proof}
Let $\Delta^\mathcal{C}_B$ be the triangular configuration from Section~\ref{sec:representation}.
Let $k_i=\left\lvert T\left(\Delta^\mathcal{C}_{b_i}\right)\right\rvert-w(b_i)$ for $i=1,\dots,d$.
We suppose that $k_1\geq k_2\geq\dots\geq k_d$, if not, we rename the indices.
Every $k_i$ is even, since $\mathcal{C}$ is even.

We expand the parts of $\Delta^\mathcal{C}_B$ by the following algorithm.
First, we define two steps. Let $\Delta$ be a triangular configuration.
Step A: We choose a triangle of $\Delta$ and subdivide it in the way depicted in Figure~\ref{fig:subdivision}. This step increases the number of triangles of $\Delta$ by $6$.
Step B: We choose two triangles of $\Delta$ connected by an edge of degree $2$ and subdivide them in the way depicted in Figure~\ref{fig:subdivision2}. This step increases the number of triangles of $\Delta$ by $4$.

We initialize the set $I$ to $\{1\}$, then we apply the following procedure while the set $I\neq\{1,\dots,d\}$.
 
Let $i$ be the smallest index not in $I$. We repeatedly apply step A to $\Delta^\mathcal{C}_{b_i}- B^n$ until $k_i\leq k_{i-1}-6$.
%If $k_i=k_{i-1}$, we put $i$ to $F$.
Then,
if $k_i=k_{i-1}-4$, we apply step B to $\Delta^\mathcal{C}_{b_i}- B^n$.
If $k_i=k_{i-1}-2$, we apply step B to $\Delta^\mathcal{C}_{b_j}- B^n$ for every $j\in I$, and step A to $\Delta^\mathcal{C}_{b_i}- B^n$. Then, we put the index $i$ to $I$ and repeat these steps.

Note that, we can apply step B to $\Delta^\mathcal{C}_{b_j}- B^n$ only if it contains two triangles connected by an edge of degree $2$. If triangular configuration $\Delta^\mathcal{C}_{b_j}- B^n$ does not contain an edge of degree $2$, we apply step A to $\Delta^\mathcal{C}_{b_i}- B^n$ for every $i=1,\dots,d$. Then, $\Delta^\mathcal{C}_{b_j}- B^n$ contains two triangles connected by an edge of degree $2$.

After this procedure, we have a balanced triangular configuration $\Delta^\mathcal{C}_B$. If $\ex{\Delta^\mathcal{C}_B}<n$, we repeatedly apply step A on $\Delta^\mathcal{C}_{b_i}- B^n$ for every $i=1,\dots,d$ unless $\ex{\Delta^\mathcal{C}_B}>n$.

\end{proof}

Let $c$ be a codeword of $\mathcal{C}$ and let $c=\sum_{b\in B}\alpha_bb$ be the unique expression of $c$, where $\alpha_b\in\mathbb{F}$. The \dfn{degree} of $c$ with respect to a basis $B$ is defined to be the number of non-zero $\alpha_b$'s.
The degree is denoted by $\vd{c}$.

\begin{prop}
\label{prop:weightf}
Let $\mathcal{C}$ be an even finite linear code over a field $\mathbb{F}$ with a representable basis $B$ and let $\Delta^\mathcal{C}_B$ be a balanced triangular configuration provided by Proposition~\ref{prop:balancedrepr} and let $f:\mathcal{C}\mapsto\ker\Delta^\mathcal{C}_B$ be the linear mapping from Definition~\ref{dfn:mappingf}. Then $w(f(c))=w(c)+\vd{c}\ex{\Delta^\mathcal{C}_B}$ for every codeword $c\in \mathcal{C}$.
\end{prop}
\begin{proof}
Write $c$ as $\sum_{b\in B}\alpha_bb$.
Then, $f(c)=\sum_{b\in B}\alpha_b\gen{\Delta^\mathcal{C}_{b}}$.
Let $I$ be the set of basis vectors $b$ such that $\alpha_b\neq 0$.
By Corollary~\ref{cor:genentries}, vector $f(c)$ has non-zero elements indexed by triangles of $\Delta^\mathcal{C}_{b}- B^n$ for all $b\in I$.
The number of these triangles is $\vd{c}\ex{\Delta^\mathcal{C}_B}$, since $\left\lvert T\left(\Delta^\mathcal{C}_{b}- B^n\right)\right\rvert=\ex{\Delta^\mathcal{C}_B}$ for all $b\in I$ and $\left\lvert I\right\rvert=\vd{c}$.
By Proposition~\ref{prop:mappingf}, $f(c)^{B^n_j}=c^j$ for $j=1,\dots,n$. Hence, the number of non-zero coordinates indexed by triangles of $B^n$ is $w(c)$.
Therefore, $w(f(c))=w(c)+\vd{c}\ex{\Delta^\mathcal{C}_B}$.
\end{proof}

%\subsection{Triangular sphere $\ce{_{[+-]}{\mathcal{S}}}^m \tensor*[^{+-}]{\mathcal{S}}{}$}
\subsection{Proof of Proposition~\ref{prop:multispherespace}}
\label{sec:proofprop}
\subsubsection{Oriented triangular sphere $\vec{\mathcal{S}^m}$}
\label{sec:triansphere}
\begin{figure}[h]
\begin{center}
	\includegraphics{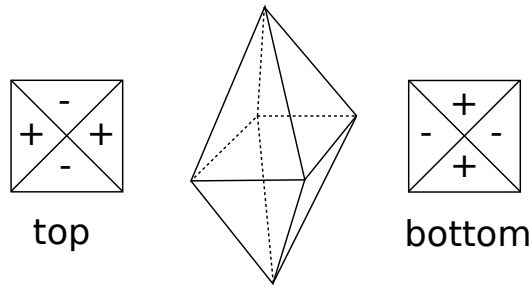}
\end{center}
\caption{Oriented triangular sphere $\vec{\mathcal{S}^8}$.}
\label{fig:sphere}
\end{figure}
In this section we give a proof of Proposition~\ref{prop:multispherespace}.
The oriented triangular sphere $\vec{\mathcal{S}^m}$ is a triangulation of a $2$-dimensional sphere by $m$ triangles, such that there is an assignment of sign '$+$' and '$-$' to triangles and every edge is incident with one '$+$' triangle and one '$-$' triangle.
An example is depicted in Figure~\ref{fig:sphere} for $m=8$.
%We denote the set of '$+$' triangles of $\vec{\mathcal{S}^m}$ by $\left[\vec{\mathcal{S}^m}\right]_{+}$ and the set of '$-$' triangles by $\left[\vec{\mathcal{S}^m}\right]_{-}$.

\begin{prop}
Let $\mathbb{F}$ be a field and let $a$ be a non-zero element of $\mathbb{F}$. Then $\ker \vec{\mathcal{S}^m}=\linspan{\{(a,-a,a,-a,\dots,a,-a)\}}$. \qed
\end{prop}

\begin{prop}
Let $k,l$ be non-negative integers. Then there exists the oriented triangular sphere $\vec{\mathcal{S}^m}$ with $m=8+6l+4k$ triangles.
\end{prop}
\begin{proof}
We construct the desired sphere $\vec{\mathcal{S}^m}$ from the sphere $\vec{\mathcal{S}^8}$, depicted in Figure~\ref{fig:sphere}, by sequentially subdividing triangles of $\vec{\mathcal{S}^8}$ in the way depicted in Figure~\ref{fig:subdivision} or \ref{fig:subdivision2}. These subdivisions increase the number of triangles by $4$ or by $6$.

\begin{figure}
 \centering
 \includegraphics{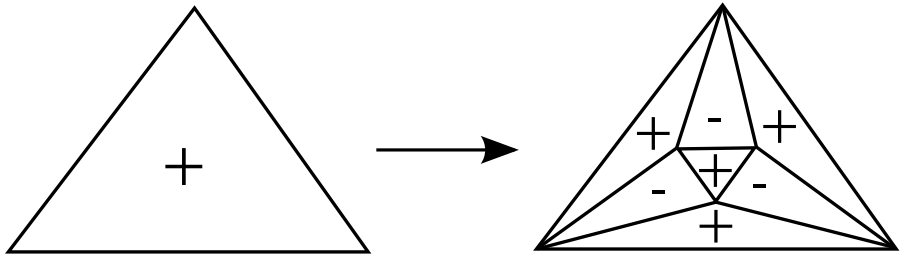}
 % subdivision.eps: 0x0 pixel, 300dpi, 0.00x0.00 cm, bb=
 \caption{Triangle subdivision}
 \label{fig:subdivision}
\end{figure}
\begin{figure}
 \centering
 \includegraphics{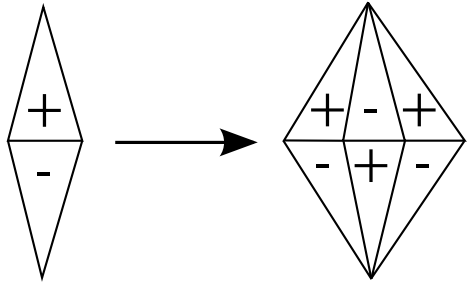}
 % subdivision2.eps: 0x0 pixel, 300dpi, 0.00x0.00 cm, bb=
 \caption{Triangles subdivision}
 \label{fig:subdivision2}
\end{figure}

\end{proof}

\subsubsection{Oriented triangular multisphere $\vec{\mathcal{M}}^M_{n_1,n_2,\dots,n_k}$}
In this subsection we construct the oriented triangular configuration which we call oriented triangular multisphere. We note that an important property of the oriented triangular multisphere which we use is that it has an even number of triangles.
In the construction of oriented triangular multisphere we proceed in four steps.

Step 1. Let $n_1,n_2,\dots,n_k$ be distinct positive integers and let $M$ be an integer.
% The oriented triangular multisphere $\vec{\mathcal{M}}_{n_1,n_2,\dots,n_k}^M$ is obtained from
%\mathcal{M}_{n_1,n_2,\dots,n_k}:=
We start with oriented triangular configuration
$$\vec{\mathcal{M}}_1:=\vec{\mathcal{S}_{1}}\cup t_{12}\cup t'_{12}\cup\vec{\mathcal{S}_{2}}\cup t_{23}\cup t'_{23}\cup\dots\cup t_{(k-1)k}\cup t'_{(k-1)k}\cup\vec{\mathcal{S}_{k}}$$
where
$t_{i(i+1)}$, $t'_{i(i+1)}$ are empty triangles, $i=1,\dots,k-1$;
and $\vec{\mathcal{S}_{j}}$ is oriented triangular sphere $\vec{\mathcal{S}^{m}}$ (see Section~\ref{sec:triansphere}) such that $m>4n_i$ and $m>2M$ for every $i=1,\dots,k$ and every $j=1,\dots,k$. If $k$ equals $1$, the triangular multisphere $\vec{\mathcal{M}}_{n_1}^M$ is $\vec{\mathcal{S}_{1}}$. Recall that an empty triangle is a set of three edges forming a boundary of a triangle.

Step 2. We make the following links between the triangles of $\vec{\mathcal{M}}_1$. For every $i=1,\dots,k-1$;
we choose $n_{i+1}$ different triangles of $\left[\vec{\mathcal{S}_{i}}\right]_-$ (for this notation see Section~\ref{sec:orientation}) and create the link by the tunnel from empty triangle $t_{i(i+1)}$ to each chosen triangle.
Then, we choose $n_i$ different triangles of $\left[\vec{\mathcal{S}}_{(i+1)}\right]_+$ and create the link by the tunnel from each chosen triangle to empty triangle $t_{i(i+1)}$.
%$n_{i+1}$ different triangles of $\left[\vec{\mathcal{S}_{i}}\right]_+$ and $n_i$ different triangles of $\left[\vec{\mathcal{S}_{(i+1)}}\right]_-$ for every $i=1,\dots,k-1$.
Then, we delete the triangles of $\left[\vec{\mathcal{S}_{i}}\right]_-$ and $\left[\vec{\mathcal{S}}_{(i+1)}\right]_+$ that we linked with a tunnel from $\vec{\mathcal{M}}_1$.
We denote the resulting triangular configuration by $\vec{\mathcal{M}}_2$.

% we added the dashed empty triangles $t'_{i(i+1)}$ then we
Step 3. To achieve the even number of triangles of the multisphere we make the following links between the triangles of $\vec{\mathcal{M}}_2$. For every $i=1,\dots,k-1$;
we choose $n_{i+1}$ different triangles of $\left[\vec{\mathcal{S}_{i}}\right]_-$ and create the link by the tunnel from empty triangle $t'_{i(i+1)}$ to each chosen triangle.
Then, we choose $n_i$ different triangles of $\left[\vec{\mathcal{S}}_{(i+1)}\right]_+$ and create the link by the tunnel from each chosen triangle to empty triangle $t'_{i(i+1)}$.
%We join empty triangle $t'_{i(i+1)}$ by tunnels with $n_{i+1}$ different triangles of $\mathcal{S}_{i+}$ and $n_i$ different triangles of $\mathcal{S}_{(i+1)-}$ for every $i=1,\dots,k-1$ .
Finally, we delete the triangles of $\left[\vec{\mathcal{S}_{i}}\right]_-$ and $\left[\vec{\mathcal{S}}_{(i+1)}\right]_+$ that we linked with a tunnel from $\vec{\mathcal{M}}_2$. The resulting triangular configuration is oriented triangular multisphere $\vec{\mathcal{M}}^M_{n_1,n_2,\dots,n_k}$.

Step 4.
For every $i=1,\dots,k-1$, we denote by $\vec{T}_i$ the triangular configuration consisting of all the tunnels linked to oriented triangular sphere $\vec{\mathcal{S}}_i$ in steps 2,3.
%Let $\vec{\mathcal{M}}$ be equal to $\vec{\mathcal{M}}_{n_1,n_2,\dots,n_k}^M$.
We denote the set of triangles $\left[\vec{\mathcal{S}_{i}}\right]_+\cup \left[\vec{T}_i\right]_+$ by $\left[\vec{\mathcal{M}}_{n_1,n_2,\dots,n_k}^M\right]_{+i}$ and the set of triangles $\left[\vec{\mathcal{S}_{i}}\right]_-\cup \left[\vec{T}_i\right]_-$ by $\left[\vec{\mathcal{M}}_{n_1,n_2,\dots,n_k}^M\right]_{-i}$.
An example of $\mathcal{M}_{n_1,n_2,\dots,n_k}^M$ is depicted in Figure~\ref{fig:multisphere}.
%We denote triangles of $\mathcal{M}_{n_1,n_2,\dots,n_k}^m$ which coincides with $\mathcal{S}^{i}_{j+(-)}$ by ${}_{j+(-)}{}^i\mathcal{M}_{n_1,n_2,\dots,n_k}^m$
\begin{figure}
 \centering
 \includegraphics[width=320pt]{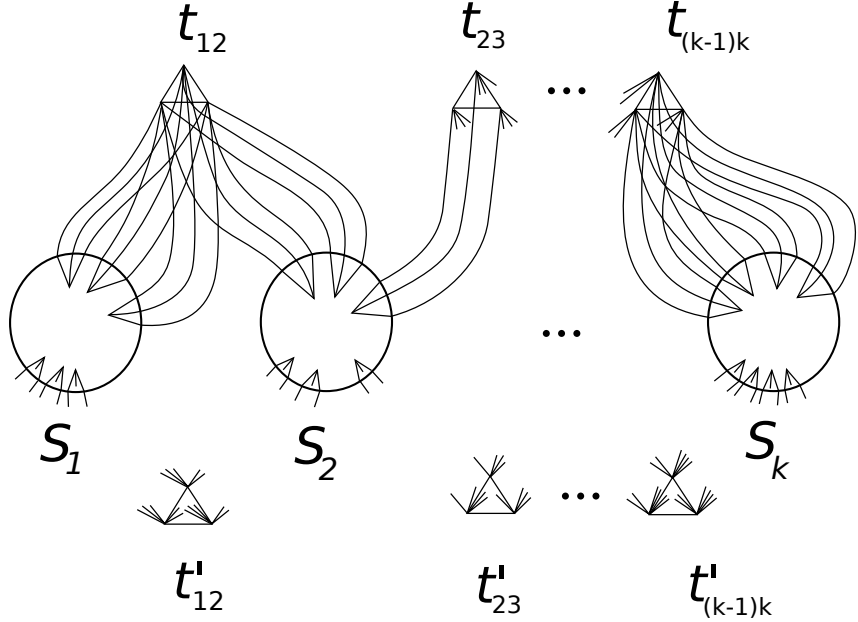}
 % multisphere.eps: 0x0 pixel, 300dpi, 0.00x0.00 cm, bb=
 \caption{Triangular multisphere $\mathcal{M}_{2,3,1,\dots,4,2}^M$}
 \label{fig:multisphere}
\end{figure}

\begin{proof}[Proof of Proposition~\ref{prop:multispherespace}]
Let $g$ be a generator of $G$ and let $n_i$ be such that $a_i=n_i\times g=\overbrace{g+g+\dots+g}^{n_i}$. We show that the desired configuration $\vec{\mathcal{M}}$ is oriented triangular multisphere $\vec{\mathcal{M}}_{n_1,n_2,\dots,n_k}^M$; where $M=n$. First, we need to show that there is a vector $v\in \ker\vec{\mathcal{M}}$ with entries containing the elements $a_1,a_2,\dots,a_k$. We construct such vector $v$ by setting coordinates indexed by the triangles of $\left[\vec{\mathcal{M}}\right]_{+i}$ (recall step 4 above) to $a_i$ and coordinates indexed by the triangles of $\left[\vec{\mathcal{M}}\right]_{-i}$ to $-a_i$ for $i=1,\dots,k$.

Now, we show that the vector $v$ belongs to $\ker \vec{\mathcal{M}}$.
Let $e$ be an edge different from the edges of empty triangles $t_{i(i+1)}$ and $t'_{i(i+1)}$. Then the edge $e$ is incident with two triangles and the equation indexed by $e$ is $a_i-a_i=0$. Let $e$ be an edge of an empty triangle $t_{i(i+1)}$ or $t'_{i(i+1)}$. Then the edge $e$ is incident with $n_i$ triangles from $\left[\vec{\mathcal{M}}\right]_{-(i+1)}$ and $n_{i+1}$ triangles from $\left[\vec{\mathcal{M}}\right]_{+i}$. So, the equation indexed by $e$ is $$n_{i+1}\times a_i-n_i\times a_{i+1}=n_{i+1}\times(n_i\times g)-n_i\times (n_{i+1}\times g)=0.$$ Hence, the vector $v$ belongs to $\ker\vec{\mathcal{M}}$.

A \dfn{triangle path} is a sequence of triangles $t_1,\dots,t_k$ such that $t_i$ and $t_{i+1}$ have a common edge, for every $i=1,\dots,k-1$. Next, we prove a claim.

\begin{clm}
\label{lem:multispheretrianglesdetermined}
 Let $e$ be an edge of $\vec{\mathcal{M}}$. Let $t_1,\dots,t_d$ be triangles incident with $e$ in any order. Let $v$ be a vector from $\ker \vec{\mathcal{M}}$. Then the entries of $v$ indexed by $t_2,\dots,t_d$ are determined by the entry indexed by $t_1$. 
\end{clm}
{\bf Proof of Claim~\ref{lem:multispheretrianglesdetermined}.}
If $e$ is neither an edge of an empty triangle $t_{i(i+1)}$ nor $t'_{i(i+1)}$, then the edge $e$ has degree $2$ and the lemma follows.

Suppose that $e$ is an edge of an empty triangle $t_{i(i+1)}$ or $t'_{i(i+1)}$. The entries indexed by the triangles that belong to $\left[\vec{\mathcal{M}}\right]_{-i}$ have the same value, since they are connected by a triangle path such that each inner edge of the triangle path has degree $2$ in $\vec{\mathcal{M}}$. The same holds for the entries indexed by the triangles of $\left[\vec{\mathcal{M}}\right]_{+(i+1)}$.
Without loss of generality we can suppose that $t_1$ belongs to $\left[\vec{\mathcal{M}}\right]_{-i}$ and let $t_l$ be an element of $\{t_2,\dots,t_d\}$ that belongs to $\left[\vec{\mathcal{M}}\right]_{+(i+1)}$.
Let $v_1$ be the entry of $v$ indexed by $t_1$ and let $v_l$ be an entry indexed by $t_l$.
If $v_1\notin G$, we choose an appropriate scalar $\alpha\in\mathbb{F}$ such that $\alpha v_1\in G$ and set $v:=\alpha v$. Then the following equation holds
\begin{equation}
\label{eq:multi}
n_{i+1}\times v_1=n_i\times v_l.
\end{equation} 
We show that there is only one solution $v_l$.
We use the fact that every cyclic subgroup $G$ of the additive group of the field $\mathbb{F}$ has a prime or an infinite order.
In the case of an infinite order, Equation~\ref{eq:multi} has only one solution $v_l$. In the case of a prime order, since the integers $n_i$ and $n_{i+1}$ do not divide the group order and by Lagrange's theorem, Equation~\ref{eq:multi} has only one solution $v_l$.
% If $a\neq0$, then $n_{i+1}\times a\neq 0$. Hence, there is exactly one solution $b$.
% If $a=0$, then $0=n_i\times b$. If the group $G$ has a prime order. We use Lagrange's theorem. Since $n_i$ does not divide the order of $G$, the value $b$ is $0$. If the group $G$ has an infinite order. We have immediately that the value $b$ is $0$.
{\bf End of proof of Claim~\ref{lem:multispheretrianglesdetermined}.}

Finally, we prove the proposition. Since there is a triangle path between any two triangles of $\vec{\mathcal{M}}$ and by the above lemma, the dimension of the kernel $\ker\vec{\mathcal{M}}$ is $1$. From the definition of the vector $v$, all entries of $v$ are non-zero.
\end{proof}
%We denote the multisphere $\vec{\mathcal{M}}$ from the above proposition by $\vec{\mathcal{M}}_{a_1,a_2,\dots,a_k}^M$ and its parts $\left[\vec{\mathcal{M}}\right]_{+i}$ and $\left[\vec{\mathcal{M}}\right]_{-i}$ by $\left[\vec{\mathcal{M}}\right]_{+a_i}$ and $\left[\vec{\mathcal{M}}\right]_{-a_i}$, respectively.

\section{Weight enumerator}
In this section, we state the connection between the weight enumerators of finite linear codes and the weight enumerators of their triangular representations. 

%This provides a proof of Theorem~\ref{thm:main}.

% \begin{prop}
% \label{prop:doubled}
% Let $\mathcal{C}$ be a binary linear code. Let $\mathcal{C}^2$ be its double code. Every codeword of $\mathcal{C}^2$ has even weight, and
% %. The following equation holds for the weight enumerators:
% \begin{equation*}
% W_\mathcal{C}(x^2)=W_{\mathcal{C}^2}(x).
% \end{equation*}
% \end{prop}
% \begin{pf}
% Every codeword $cc$ of $\mathcal{C}^2$ has even weight, since $w(cc)=2w(c)$ for some codeword $c$ of $\mathcal{C}$.
% \begin{equation*}
% W_\mathcal{C}(x^2)=\sum_{c\in \mathcal{C}} (x^2)^{w(c)}=\sum_{cc\in \mathcal{C}^2} x^{w(cc)}=W_{\mathcal{C}^2}(x).
% \qed
% \end{equation*}
% \end{pf}
% \begin{cor}
% The double code of a binary linear code has a balanced triangular representation.
% \end{cor}

We define the \dfn{extended weight enumerator} (with respect to a fixed basis) by
\begin{equation*}
W^k_\mathcal{C}(x):=\sum_{\substack{c\in \mathcal{C}\\\vd{c}=k}} x^{w(c)}.
\end{equation*}

If a code $\mathcal{C}$ has dimension $d$, then
\begin{equation*}
W_\mathcal{C}(x)=\sum_{k=0}^d W^k_\mathcal{C}(x).
\end{equation*}

Recall that a basis $B$ is representable if for each $b\in B$ each entry of $b$ belongs to a cyclic subgroup $G(b)$ of the additive group of the field $\mathbb{F}$. A code $\mathcal{C}$ is even if all codewords have an even weight. 

\begin{prop}
\label{prop:extended polynomial}
Let $\mathcal{C}$ be an even finite linear code with a representable basis $B$ and let $\Delta^\mathcal{C}_B$ be a balanced triangular configuration provided by Proposition~\ref{prop:balancedrepr}. Then
\begin{equation*}
W^k_{\ker\Delta^\mathcal{C}_B}(x)=W^k_\mathcal{C}(x)x^{k\ex{\Delta^\mathcal{C}_B}}.
\end{equation*}
\end{prop}
\begin{proof}
Let $f$ be the mapping from Definition~\ref{dfn:mappingf}.
For every codeword $c$ of degree $k$ of $\mathcal{C}$ there is codeword $f(c)$ of degree $k$ of $\ker\Delta^\mathcal{C}_B$.
By Proposition~\ref{prop:weightf}, $w(f(c))=w(c)+k\ex{\Delta^\mathcal{C}_B}$.
Therefore,
\begin{equation*}
\begin{split}
W^k_{\ker\Delta^\mathcal{C}_B}(x)& =\sum_{\substack{f(c)\in \ker\Delta^\mathcal{C}_B\\\vd{f(c)}=k}} x^{w(f(c))} =\sum_{\substack{c\in \mathcal{C}\\\vd{c}=k}} x^{w(c)+k\ex{\Delta^\mathcal{C}_B}}
 =W^k_{\mathcal{C}}(x)x^{k\ex{\Delta^\mathcal{C}_B}}.
\end{split}
\end{equation*}
\end{proof}
%Now we show how to compute weight polynomial $W_\mathcal{C}(x)$ of a code $\mathcal{C}$ from the weight polynomial $W_{\Delta^\mathcal{C}_B}(x)$ of its triangular representation $\Delta^\mathcal{C}_B$.
\begin{prop}
\label{prop:weights}
Let $\mathcal{C}$ be an even finite linear code of length $n$ with a representable basis $B$ and let $\Delta^\mathcal{C}_B$ be a balanced triangular configuration provided by Proposition~\ref{prop:balancedrepr}.
Then the inequality $k\ex{\Delta^\mathcal{C}_B}\leq w(d)\leq k\ex{\Delta^\mathcal{C}_B}+n$ holds for every codeword $d$ of degree $k$ of $\ker\Delta^\mathcal{C}_B$.
\end{prop}
\begin{proof}
Let $f$ be the mapping from Definition~\ref{dfn:mappingf}.
By Proposition~\ref{prop:weightf}, $w(d)=w(f^{-1}(d))+k\ex{\Delta^\mathcal{C}_B}$.
Since $0\leq w(f^{-1}(d))\leq n$ for every $d\in \ker\Delta^\mathcal{C}_B$, the inequality $k\ex{\Delta^\mathcal{C}_B}\leq w(d)\leq k\ex{\Delta^\mathcal{C}_B}+n$ holds.
\end{proof}

\begin{cor}
Let $\mathcal{C}$ be an even finite linear code of dimension $d$ and length $n$ with a representable basis $B$ and let $\Delta^\mathcal{C}_B$ be a balanced triangular configuration provided by Proposition~\ref{prop:balancedrepr} such that $n<\ex{\Delta^\mathcal{C}_B}$. Denote $\ex{\Delta^\mathcal{C}_B}$ by $e$. Let $\sum_{i=0}^{de+n}a_ix^i$ be the weight enumerator of $\ker\Delta^\mathcal{C}_B$. Then
\begin{equation*}
W^k_{\ker\Delta^\mathcal{C}_B}(x)=\sum_{i=ke}^{ke+n}a_ix^i. \qed
\end{equation*}
\end{cor}
% \begin{proof}
% By Proposition~\ref{prop:weights}, $w(c)\leq(k-1)e+n$ for all codewords $c\in \ker\Delta^\mathcal{C}_B$ of a degree less than $k$. Since $n<e$, the inequality $w(c)\leq ke-e+n<ke$ holds.
% By Proposition~\ref{prop:weights}, $(j+1)e\leq w(c)$ for all codewords $c\in \ker\Delta^\mathcal{C}_B$ of a degree greater than $k$. Since $n<e$, the inequality $ke+e<ke+n\leq w(c)$ holds.
% Hence, the enumerator $W^k_{\ker\Delta^\mathcal{C}_B}(x)$ is the sum over all codewords of a weight between $ke$ and $ke+n$.
% \end{proof}

\begin{thm}
\label{thm:weight polynomial}
Let $\mathcal{C}$ be an even finite linear code of dimension $d$ and length $n$ with a representable basis $B$ and let $\Delta^\mathcal{C}_B$ be a balanced triangular configuration provided by Proposition~\ref{prop:balancedrepr} such that $n<\ex{\Delta^\mathcal{C}_B}$. Denote $\ex{\Delta^\mathcal{C}_B}$ by $e$. Let $\sum_{i=0}^{de+n}a_ix^i$ be the weight polynomial of $\ker\Delta^\mathcal{C}_B$. Then
\begin{equation*}
W_\mathcal{C}(x)=\sum_{i=0}^{de+n}a_ix^{i\bmod e}.
\end{equation*}
\end{thm}
\begin{proof}
The inequality $w(c)\leq n$ holds for every codeword $c\in \mathcal{C}$.
Let $f$ be the mapping from Definition~\ref{dfn:mappingf}.
By Proposition~\ref{prop:weightf}, $w(f(c))=w(c)+\vd{c}e$ for every codeword $c$ of $\mathcal{C}$.
Since $n<e$, the following equality holds.
\begin{equation*}
w(f(c))\bmod e = (w(c)+\vd{c}e)\bmod e = w(c).
\end{equation*}
Hence,
\begin{equation*}
W_\mathcal{C}(x)=\sum_{i=0}^{de+n}a_ix^{i\bmod e}.
\end{equation*}
\end{proof}

\subsection{Proof of Theorem~\ref{thm:repr1}}
The \dfn{double code}, denoted by $\mathcal{C}^2$, of a linear code $\mathcal{C}$ of length $n$ is the code $$\mathcal{C}^2=\left\{\left(c^1,\dots,c^n,c^1,\dots,c^n\right) : c\in C\right\}.$$
\begin{prop}
\label{prop:double}
Let $\mathcal{C}$ be a linear code and let $\mathcal{C}^2$ be its double code. Then, the double code $\mathcal{C}^2$ is even and the code $\mathcal{C}$ is a punctured code of its double code $\mathcal{C}^2$ and there is a linear bijection between $\mathcal{C}$ and $\mathcal{C}^2$ that maps minimal codewords to minimal codewords and $W_\mathcal{C}(x)=W_{\mathcal{C}^2}(x^{\frac{1}{2}})$.\qed
\end{prop}

\begin{proof}[Proof of Theorem~\ref{thm:repr1}.]
Let $\mathcal{C}^2$ be the double code of $\mathcal{C}$. The code $\mathcal{C}^2$ is even. Let $B^2$ be the basis $\{(b^1,\dots,b^n,b^1,\dots,b^n)\vert b\in B\}$ of $\mathcal{C}^2$. The basis $B^2$ is representable. Let $\Delta^{\mathcal{C}^2}_{B^2}$ be a balanced triangular configuration provided by Proposition~\ref{prop:balancedrepr} such that $\ex{\Delta^{\mathcal{C}^2}_{B^2}}>n_2$, where $n_2=2n$ is the length of $\mathcal{C}^2$. We denote $\ex{\Delta^{\mathcal{C}^2}_{B^2}}$ by $e$.
By Theorem~\ref{thm:trianrep}, the code $\mathcal{C}^2$ is equal to $\ker\Delta^{\mathcal{C}^2}_{B^2}/S'$ for some set of indices $S'$ and there exists a linear bijection $f':\mathcal{C}^2\mapsto\ker\Delta^{\mathcal{C}^2}_{B^2}$ which maps minimal codewords to minimal codewords.
By Proposition~\ref{prop:double}, the code $\mathcal{C}$ is equal to $\mathcal{C}^2/S''$ where $S''=\{n+1,\dots,2n\}$ and there is a linear bijection $f'':\mathcal{C}\mapsto\mathcal{C}^2$ which maps minimal codewords to minimal codewords. Therefore, the code $\mathcal{C}$ is equal to $\ker\Delta^{\mathcal{C}^2}_{B^2}/(S'\cup S'')$ and there is a linear bijection $f:\mathcal{C}\mapsto\ker\Delta^{\mathcal{C}^2}_{B^2}$ which maps minimal codewords to minimal codewords. Hence, the code $\mathcal{C}$ is triangular representable and $\Delta^{\mathcal{C}^2}_{B^2}$ is its triangular representation. We denote $\Delta^{\mathcal{C}^2}_{B^2}$ by $\Delta$.

By Proposition~\ref{prop:expcont}, $e=\lvert S'\rvert/\dim\mathcal{C}^2$.
Let $S(\ker\Delta,\mathcal{C})=S'\cup S''$. The cardinality $\lvert S''\rvert$ is equal to $n$. Hence, $e=(\lvert S(\ker\Delta,\mathcal{C})\rvert-n)/\dim\mathcal{C}^2=(\lvert S(\ker\Delta,\mathcal{C})\rvert-n)/\dim\mathcal{C}$.

If the code $\mathcal{C}$ is finite, the formula for the weight enumerator follows from Theorem~\ref{thm:weight polynomial} and from Proposition~\ref{prop:double}.

\end{proof}
\begin{proof}[Proof of Corollary~\ref{thm:repr2}.]
The additive group of $GF(p)$, where $p$ is a prime, is cyclic. In the case of rationals, we multiply the basis vectors by a sufficiently large integer, so that all vectors are integral. Hence, all elements of all basis vectors belong to the cyclic group of integers. Then, we use Theorem~\ref{thm:repr1}.
\end{proof}

\subsection{Multivariate weight enumerator}
\label{sec:multi}

Let $\mathcal{C}$ be a linear code of length $n$. Let $R$ be the set of coordinates of $\mathcal{C}$. The \dfn{assignment of variables to coordinates} is a mapping $\lambda$ from $R$ to the set of indices of variables $\{1,\dots,k\}$. The \dfn{multivariate weight enumerator} of $\mathcal{C}$ is $$W_\mathcal{C}^\lambda(x_1,\dots,x_k)=\sum_{c\in\mathcal{C}}\prod_{\substack{i=1\\c_i\neq 0}}^nx_{\lambda(i)}.$$

\begin{thm}
\label{thm:multrepr1}
Let $\mathcal{C}$ be a finite linear code over a field $\mathbb{F}$, $\mathcal{C}\subseteq\mathbb{F}^n$. Let $\lambda$ be an assignment of variables.
If $\mathcal{C}$ is triangular representable, then there exists a triangular configuration $\Delta$ and an injection $\mu:\{1,\dots,2n\}\mapsto T(\Delta)$
% an assignment of variables $\lambda'$ 
such that:
%Let $\lambda$ be an arbitrary assignment of variables such that $\lambda(j)=k$ for all $j\in S(\Delta)$.
if $$\sum_{\substack{i_1+i_2+\dots+i_k\leq m\\i_1\geq0,i_2\geq0,\dots,i_k\geq0}}a_{i_1i_2\dots i_k}x_1^{i_1}x_2^{i_2}\dots x_k^{i_k}$$ is the multivariate weight enumerator $W_{\ker\Delta}^{\lambda'}(x_1,\dots,x_k)$ of $\ker\Delta$ with the assignment of variables $\lambda'$ defined: $\lambda'(t):=\lambda(\mu^{-1}(t))$ if $t\in\mu(\{1,\dots,n\})$ and $\lambda'(t):=\lambda(\mu^{-1}(t)-n)$ if $t\in\mu(\{n+1,\dots,2n\})$ and $\lambda'(t):=k$ if $t\notin\mu(\{1,\dots,2n\})$.
Then
\begin{equation*}
W_\mathcal{C}^\lambda(x_1,\dots,x_k)=\sum_{\substack{i_1+i_2+\dots+i_k\leq m\\i_1\geq0,i_2\geq0,\dots,i_k\geq0}}a_{i_1i_2\dots i_k}x_1^{(i_1/2)}x_2^{(i_2/2)}\dots x_{k-1}^{(i_{k-1}/2)}x_k^{((i_k\bmod e)/2)},
\end{equation*}
where $e=(\lvert S(\ker\Delta,\mathcal{C})\rvert-n)/\dim\mathcal{C}$.
\end{thm}
\begin{proof}
By Theorem~\ref{thm:nonrepr}, the code $\mathcal{C}$ has a representable basis $B$.
Let $\mathcal{C}^2$ be the double code of $\mathcal{C}$. Let $B^2$ be the basis $\{(b^1,\dots,b^n,b^1,\dots,b^n)\vert b\in B\}$ of $\mathcal{C}^2$. The code $\mathcal{C}^2$ is even and the basis $B^2$ is representable. Let $\Delta^{\mathcal{C}^2}_{B^2}$ be a balanced triangular configuration provided by Proposition~\ref{prop:balancedrepr} such that $\ex{\Delta^{\mathcal{C}^2}_{B^2}}>n_2$, where $n_2=2n$ is the length of $\mathcal{C}^2$. The desired configuration $\Delta$ is $\Delta^{\mathcal{C}^2}_{B^2}$. We denote $\ex{\Delta}$ by $e$.
By Theorem~\ref{thm:trianrep}, the code $\mathcal{C}^2$ is equal to $\ker\Delta/S'$ for some set of indices $S'$.
By Proposition~\ref{prop:double}, the code $\mathcal{C}$ is equal to $\mathcal{C}^2/S''$ where $S''=\{n+1,\dots,2n\}$. Therefore, the code $\mathcal{C}$ is equal to $\ker\Delta/(S'\cup S'')$. 
By Proposition~\ref{prop:expcont}, $e=\lvert S'\rvert/\dim\mathcal{C}^2$.
Let $S(\ker\Delta,\mathcal{C})=S'\cup S''$. The cardinality $\lvert S''\rvert$ is equal to $n$. Hence, $e=(\lvert S(\ker\Delta,\mathcal{C})\rvert-n)/\dim\mathcal{C}^2=(\lvert S(\ker\Delta,\mathcal{C})\rvert-n)/\dim\mathcal{C}$.

We define an assignment of variables of $\mathcal{C}^2$ in the following way: $\lambda''(i):=\lambda''(n+i):=\lambda(i)$ for $i=1,\dots,n$.
Let $\mu:\{1,\dots,2n\}\mapsto T(\Delta)$ be the injection from Definition~\ref{dfn:mappingf}. Then, $S'=T(\Delta)\setminus\mu(\{1,\dots,2n\})$.
% We define the asignment of variables $\lambda':T(\Delta):\mapsto\{1,\dots,k\}$ as $\lambda'(t):=k$ if $t\in S(\ker\Delta,\mathcal{C})$ and $\lambda'(t):=\lambda''(\mu^{-1}(t))$ if $t\notin S(\ker\Delta,\mathcal{C})$.

% We construct a balanced triangular representation of $\mathcal{C}$ as it is described in Section~\ref{sec:representation} and in Proposition~\ref{prop:balancedrepr}.
% By Proposition~\ref{prop:balancedrepr}, balanced representation exists if the code is even. To ensure this we first make representation of the double code $\mathcal{C}^2$ instead of given code $\mathcal{C}$. 
% 
% Let $\Delta$ be a balanced triangular representation of $\mathcal{C}^2$.
% Then, $\mathcal{C}^2=\Delta/S''$. By definition of $e(\Delta)$, the following equation holds $e(\Delta)=\lvert S'\rvert/\dim\mathcal{C}$.
%  Since $\mathcal{C}=\mathcal{C}^2/S'$, we have $\mathcal{C}=\Delta/(S'\cup S'')$. There is a bijection between $\mathcal{C}$ and $\mathcal{C}^2$ and a bijection between $\mathcal{C}^2$ and $\ker\Delta$, both map minimal codewords to minimal codewords. Hence, there is a bijection between $\mathcal{C}$ and $\ker\Delta$ that maps minimal codewords to minimal codewords. Thus $\Delta$ is a representation of $\mathcal{C}$.
% Denote $(S'\cup S'')$ by $S$. Then, $e(\Delta)=(\lvert S\rvert-n)/\dim\mathcal{C}$.
% We define assignment of variables of $\Delta$ in the following way: $\lambda'(i):=\lambda''(i)$ for $i=1,\dots,2n$ and $\lambda'(i):=k$ for all $i\in S$.

The weight of the variable indexed by $j$ in a codeword $c$ with respect to an assignment of variables $\lambda$ is the number of nonzero coordinates of $c$ assigned to variable $j$. We denote this weight by $w^\lambda_j(c)$. Let $f:\mathcal{C}^2\mapsto\ker\Delta$ be the mapping from Definition~\ref{dfn:mappingf}.
By Proposition~\ref{prop:mappingf}, $f(c)^t=c^{\mu^{-1}(t)}$ for $t\in\mu^{-1}(\{1,\dots,2n\})$ and $c\in\mathcal{C}^2$.
For $j=1,\dots,k-1$ we have $w^{\lambda'}_j(f(c))=w^{\lambda''}_j(c)$, since $\lambda'(t)=k\neq j$ for all $t\in S'=T(\Delta)\setminus\mu(\{1,\dots,2n\})$ and $\lambda'(t)=\lambda''(\mu^{-1}(t))$ for $t\in\mu^{-1}(\{1,\dots,2n\})$. The Hamming weight of $c$ satisfies $w(c)=\sum_{j=1}^kw^\lambda_j(c)$ for arbitrary $\lambda$.
By Proposition~\ref{prop:weightf}, $w(f(c))=w(c)+\vd{c}e$ for every codeword $c$ of $\mathcal{C}^2$.
From this equation we subtract equations $w^{\lambda'}_j(f(c))=w^{\lambda''}_j(c)$ for $j=1,\dots,k-1$. Hence, $w^{\lambda'}_k(f(c))=w^{\lambda''}_k(c)+\vd{c}e$.
The inequality $w(c)\leq 2n$ holds for every codeword $c\in \mathcal{C}^2$.
Since $2n<e$, the following equality holds.
\begin{equation*}
w^{\lambda'}_k(f(c))\bmod e = (w^{\lambda''}_k(c)+\vd{c}e)\bmod e = w^{\lambda''}_k(c).
\end{equation*}

The weights of codewords of a code and its double code satisfy $2w^{\lambda}_j((c^1,\dots,c^n))=w^{\lambda''}_j((c^1,\dots,c^n,c^1,\dots,c^n))$ for every $j=1,\dots,k$.
Hence,
\begin{equation*}
W_\mathcal{C}^\lambda(x_1,\dots,x_k)=\sum_{\substack{i_1+i_2+\dots+i_k\leq m\\i_1\geq0,i_2\geq0,\dots,i_k\geq0}}a_{i_1i_2\dots i_k}x_1^{(i_1/2)}x_2^{(i_2/2)}\dots x_{k-1}^{(i_{k-1}/2)}x_k^{((i_k\bmod e)/2)},
\end{equation*}

\end{proof}

\begin{thm}
\label{cor:multpotts}
Let $G$ be a connected graph and let $n$ be the number of edges of $G$. Let $E=\{e_1,\dots,e_n\}$ be the set of edges of $G$ and let $w:E\mapsto\{w_1,\dots,w_k\}$ be weights of edges of $G$. Let $q$ be a prime. Then there exists triangular configuration $\Delta$ and an injection $\mu:\{1,\dots,2n\}\mapsto T(\Delta)$ such that if $$\sum_{\substack{i_1+i_2+\dots+i_k\leq m\\i_1\geq0,i_2\geq0,\dots,i_k\geq0}}a_{i_1i_2\dots i_k}x_1^{i_1}\dots x_{k-1}^{i_{k-1}}x_k^{i_k}$$
is the multivariate weight enumerator $W^{\lambda'}_{\ker\Delta}(x_1,\dots,x_k)$ of $\ker\Delta$ over $GF(q)$ with the assignment of variables $\lambda'$ defined:
$\lambda'(t):=i$ if $t\in\mu(\{1,\dots,n\})$ and $w(e_{\mu^{-1}(t)})=w_i$ or $t\in\mu(\{n+1,\dots,2n\})$ and $w(e_{\mu^{-1}(t)-n})=w_i$, and $\lambda'(t):=k$ if $t\notin\mu(\{1,\dots,2n\})$, then $P^q(G,x)$ equals
\begin{equation*}
\frac{q}{\prod_{uv\in E} x^{-w(uv)}}\sum_{\substack{i_1+i_2+\dots+i_k\leq m\\i_1\geq0,i_2\geq0,\dots,i_k\geq0}}a_{i_1i_2\dots i_k}x^{-w_1(i_1/2)}\dots x^{-w_{k-1}(i_{k-1}/2)}x^{-w_k((i_k\bmod e)/2)},
\end{equation*}
where $e$ is a positive integer linear in number of edges of $G$.
\end{thm}
\begin{proof}
 The proof follows from the following calculations.
\begin{equation*}
 \prod_{uv\in E} x^{-w(uv)} P^q(G,x)=\sum_{s:V\mapsto \{0,\dots,q-1\}}\prod_{uv\in E}x^{\delta(s(u),s(v))w(uv)-w(uv)}
\end{equation*}
Let $cut(s)$ be the set of edges $uv$ of $G$ such that $s(u)\neq s(v)$. Then

\begin{equation*} 
\sum_{s:V\mapsto \{0,\dots,q-1\}}\prod_{uv\in E}x^{\delta(s(u),s(v))w(uv)-w(uv)}=\sum_{s:V\mapsto \{0,\dots,q-1\}}\prod_{uv\in cut(s)}x^{-w(uv)}.
\end{equation*}
The following part of this proof is taken from Galluccio and Loebl~\cite{loeblpfaffian2} and modified for the multivariate enumerator.
Let $z$ be a vector from $GF(q)^{\lvert V(G)\rvert}$ defined as $z_v:=s(v)$. Let $O_G$ be an oriented incidence matrix of $G$.
Then, $e\in cut(s)$ if and only if $(O_g^Tz)_e\neq 0$. Hence,
\begin{equation*}
\sum_{s:V\mapsto \{0,\dots,q-1\}}\prod_{uv\in cut(s)}x^{-w(uv)}=\sum_{z\in GF(q)^{\vert V(G)\rvert}}\prod_{\substack{(O^Tz)_{uv}\neq 0\\ uv\in E(G)}}x^{-w(uv)}.
\end{equation*}
Let $\mathcal{C}$ equals $\{O^T_Gz\vert z\in GF(q)\}$. Let us define an equivalence on $GF(q)^{\lvert V(G)\rvert}$ by $w\equiv z$ if $O^T_Gw=O^T_Gz$. Observe that each equivalence class consists of $q$ elements since $O^T_Gw=O^T_Gz$ if and only if $z-w$ is a constant vector. Hence, 
\begin{equation*}
\begin{split}
\sum_{z\in GF(q)^{\vert V(G)\rvert}}\prod_{\substack{(O^Tz)_{uv}\neq 0\\ uv\in E(G)}}x^{-w(uv)}&=q\sum_{c\in\mathcal{C}}\prod_{\substack{c_{uv}\neq 0\\uv\in E(G)}}x^{-w(uv)}=q\sum_{c\in\mathcal{C}}\prod_{\substack{c_{uv}\neq 0\\uv\in E(G)}}x^{-w_{\lambda(uv)}}\\&=qW^{\lambda}_\mathcal{C}(x^{-w_1},\dots,x^{-w_k}),
\end{split}
\end{equation*}
where $\lambda$ is the assignment of variables defined: $\lambda(uv):=i$ if $w(uv)=w_i$, for $i=1,\dots,k$.
By Corollary~\ref{thm:repr2}, $\mathcal{C}$ is triangular representable. By Theorem~\ref{thm:multrepr1},
\begin{equation*}
qW_\mathcal{C}^{\lambda}(x^{-w_1},\dots,x^{-w_k})=q\sum_{\substack{i_1+i_2+\dots+i_k\leq m\\i_1\geq0,i_2\geq0,\dots,i_k\geq0}}a_{i_1i_2\dots i_k}x^{-w_1(i_1/2)}\dots x^{-w_{k-1}(i_{k-1}/2)}x^{-w_k((i_k\bmod e)/2)},
\end{equation*}
where $e=(\lvert S(\ker\Delta,\mathcal{C})\rvert-\lvert E(G)\rvert)/\dim\mathcal{C}$.

\end{proof}
Theorem~\ref{cor:potts} follows from the above theorem.

\section{Triangular non-representability}

In this section we prove a sufficient condition for triangular non-re\-pre\-sen\-ta\-bi\-li\-ty. We will show that for non-prime fields every construction of triangular representation fails on very weak condition that a linear code and its triangular representation have to have the same dimension.

\begin{proof}[Proof of Theorem~\ref{thm:nonrepr}.]
First, we observe that the field $\mathbb{F}$ contains a proper subfield $\mathbb{P}$.
We use the fact that every field contains a prime subfield $\mathbb{P}$ isomorphic to rationals or $GF(q)$, where $q$ is a prime, and the fact that every two elements of a prime field belong to a common cyclic subgroup of the additive group of the prime field.
%So, we need to show that $\mathbb{F}$ is neither rationals nor $GF(q)$, where $q$ is a prime.
%Since there exist two elements that do not belong to the same cyclic subgroup of the additive group of $\mathbb{F}$, the field $\mathbb{F}$ is different from rationals and $GF(q)$, where $q$ is a prime. 

%A prime field $\mathbb{P}$ contains for every two elements a cyclic subgroup of a prime or an infinite order of the additive group that contains both elements.
%Therefore the field $\mathbb{F}$ contains a proper prime subfield $\mathbb{P}$.
We can view the field $\mathbb{F}$ as a vector space over $\mathbb{P}$. This space can have an infinite dimension. An element $f$ of $\mathbb{F}$ is equal to $(f^1,f^2,\dots)$, where $f^i$'s are elements of $\mathbb{P}$.
We identify the vectors $(f^1,0,0,\dots)$ that have only first non-zero coordinate with the subfield $\mathbb{P}$.
Let $f=(f^1,f^2,f^3,\dots)$ be an element of $\mathbb{F}$. We define two projections on the vector space $\mathbb{F}$ over $\mathbb{P}$:
$$[f]^1:=\left(f^1,0,0,\dots\right)$$ and
$$[f]^{2+}:=\left(0,f^2,f^3,\dots\right).$$
Note that, $[f]^1\in\mathbb{P}$ for all $f\in\mathbb{F}$, and $[f]^{2+}=0$ or $[f]^{2+}\in\mathbb{F}\setminus\mathbb{P}$ for all $f\in\mathbb{F}$.

For a contradiction suppose that the linear code $\mathcal{C}$ is triangular representable.
We say that a vector of $\mathcal{C}$ is bad if it has two entries which belong to no common cyclic subgroup.
Let $\Delta$ be a triangular representation of $\mathcal{C}$.
Let $B$ be a basis of $\mathcal{C}$ with the minimum number of bad vectors.
Let $b\in B$ be a bad vector.
We recall that each basis $B$ has a bad vector by assumptions of the theorem.

%Let $b$ be equal to $(p,f,\dots)$.
The bad vector $b$ contains two entries $p,f$ which belong to no common cyclic subgroup.
We can suppose that $p$ belongs to $\mathbb{P}$, otherwise we choose a non-zero scalar $\alpha$ such that $\alpha p\in\mathbb{P}$ and replace $b$ by $\alpha b$ in basis $B$.
Then, the entry $f$ belongs to $\mathbb{F}\setminus\mathbb{P}$.

Let $v$ be a vector from $\ker\Delta$. We define two projections:
$$[v]^1:=([v^1]^1,[v^2]^1,[v^3]^1,\dots)$$
and
$$[v]^{2+}:=([v^1]^{2+},[v^2]^{2+},[v^3]^{2+},\dots).$$
Since every element of the incidence matrix of $\Delta$ is $0$ or $1$, the projections $[v]^1$ and $[v]^{2+}$ of every vector $v\in\ker\Delta$ belong to $\ker\Delta$.

Since the linear code $\mathcal{C}$ is a punctured code of $\ker\Delta$ and the codes $\mathcal{C}$ and $\ker\Delta$ have the same number of codewords, we can define mapping $g$ from $\mathcal{C}$ to $\ker\Delta$ such that $c$ is a punctured codeword of $g(c)$ for every $c\in\mathcal{C}$.
The mapping $g$ is linear and bijective.

%we can suppose that $g$ is a linear mapping such that $g(b)=(p,f,\dots)$.
The set $B_\Delta:=\{g(b)\vert b\in B\}$ is a basis of $\ker\Delta$.
The equation $g(b)=[g(b)]^1+[g(b)]^{2+}$ holds.
Since $g(b)$ contains both entries $p$ and $f$, both vectors $[g(b)]^1$ and $[g(b)]^{2+}$ are non-zero. The vectors $[g(b)]^1$ and $g[(b)]^{2+}$ are linear independent. Hence, one of the vectors $[g(b)]^1$ or $[g(b)]^{2+}$ does not belong to $\linspan{B_\Delta\setminus\{g(b)\}}$. We denote this vector by $g(b')$.

Hence, the set $\{g(b')\}\cup B_\Delta\setminus \{g(b)\}$ is a basis of $\ker\Delta$ and the set $B':=\{b'\}\cup B\setminus \{b\}$ is a basis of $\mathcal{C}$.
Now, basis $B'$ has smaller number of bad vectors than $B$, a contradiction with the minimality of $B$.

\end{proof}

\begin{proof}[Proof of Corollary~\ref{cor:nonrepr}.]
We know that the field $\mathbb{F}$ contains a proper subfield $\mathbb{P}$.
Let $p$ be an element of $\mathbb{P}$ and let $f$ be an element of $\mathbb{F}\setminus\mathbb{P}$. By Theorem~\ref{thm:nonrepr}, the linear code $C=\linspan{\{(f,p)\}}$ over $\mathbb{F}$ is not triangular representable.
\end{proof}

\section*{Acknowledgment}
I would like to thank Martin Loebl for helpful discussions and continual support.

\bibliography{references}{}

\begin{thebibliography}{1}
\expandafter\ifx\csname url\endcsname\relax
  \def\url#1{\texttt{#1}}\fi
\expandafter\ifx\csname urlprefix\endcsname\relax\def\urlprefix{URL }\fi
\expandafter\ifx\csname href\endcsname\relax
  \def\href#1#2{#2} \def\path#1{#1}\fi

\bibitem{loeblpfaffian}
A.~Galluccio, M.~Loebl, On the theory of {P}faffian orientations. {I}.
  {P}erfect matchings and permanents, Electron. J. Combin. 6~(R6) (1999) .

\bibitem{cimasoni-2007}
D.~Cimasoni, N.~Reshetikhin, Dimers on surface graphs and spin structures. {I},
  Comm. Math. Phys. 275~(1) (2007) 187--208.

\bibitem{rytir2}
P.~Rytíř, Geometric representations of binary codes and computation of weight
  enumerators, Adv. in Appl. Math. 45 (2010) 290--301.

\bibitem{loeblpfaffian2}
A.~Galluccio, M.~Loebl, On the theory of {P}faffian orientations. {I}{I}.
  t-joins, k-cuts, and duality of enumeration., Electron. J. Combin. 6~(R7)
  (1999) .

\bibitem{loeblarf}
M.~Loebl, G.~Masbaum, On the optimality of the {A}rf invariant formula for
  graph polynomials, Adv. Math. 226 (2011) 332--349.

\bibitem{vanderwarder}
B.~van~der Warden, Die lange reichweite der regelmässigen atomanordnung in
  mischkristallen, Z. Phys. 118 (1941) 473.

\bibitem{loeblbook}
M.~Loebl, Discrete Mathematics in Statistical Physics, Vieweg+Teubner, 2010.

\bibitem{sokal}
A.~D. Sokal, The multivariate tutte polynomial (alias potts model) for graphs
  and matroids, Surveys in Combinatorics (2005) 173--226.

\end{thebibliography}
\bibliographystyle{elsarticle-num}
\end{document}